\documentclass{article}
\usepackage[english]{babel}
\usepackage{amsmath}
\usepackage{amsfonts}
\usepackage{amsthm}

\newcommand{\cf}{{\mathrm{cf}}}
\newcommand{\FF}{{\mathrm{ff}}}
\newcommand{\PM}{{\mathrm{pm}}}
\renewcommand{\vec}[1]{{\ensuremath{\boldsymbol{\mathrm #1}}}}
\newcommand{\ten}[1]{\ensuremath{\boldsymbol{\mathsf{#1}}}}

\newcommand{\vdot}{\boldsymbol{\mathsf{\ensuremath\cdot}}}

\newcommand{\RR}{\mathbb{R}}
\newcommand{\diver}{\ensuremath{\operatorname{div}}}
\usepackage{stmaryrd}
\usepackage{mathtools}
\DeclarePairedDelimiter{\norm}{\lVert}{\rVert}
\newcommand{\comment}[1]{}

\usepackage{hyperref}
\usepackage{cleveref}

\newtheorem{corollary}{Corollary}

\title{Effective coupling conditions for arbitrary flows in Stokes--Darcy systems}
\author{Elissa         
    Eggenweiler and
	Iryna Rybak\footnote{University of Stuttgart, Institute of Applied Analysis and Numerical Simulation, Pfaffenwaldring 57, 70569 Stuttgart, emails: elissa.eggenweiler@ians.uni-stuttgart.de, rybak@ians.uni-stuttgart.de} \\
	 \\
	}
\vspace{-100ex}
\date{ }
\begin{document}

\maketitle

\begin{abstract}
Boundary conditions at the interface between the free-flow region and the adjacent porous medium is a key issue for physically consistent model-ing and accurate numerical simulation of flow and transport processes in coupled systems due to the interface driven nature of such processes. 
Interface conditions available in the literature have several weak points: most of them are suitable only for flows parallel to the fluid--porous interface, some are restricted to specific boundary value problems, and  others contain unknown model parameters which still need to be deter-mined. These facts severely restrict the variety of applications that can be successfully modeled. We propose new interface conditions which are valid for arbitrary flow directions. These coupling conditions are rigorously derived using the theory of homogenization and boundary layers. All effective parameters appearing in these conditions are computed numeri-cally based on the geometrical configuration of the coupled system. The developed conditions are validated by comparison of numerical simulation  
results for the coupled Stokes--Darcy model and  the pore-scale resolved 
model. In addition, the new interface conditions are compared with the classical conditions to demonstrate the advantage of the proposed conditions. 
\end{abstract}
\vspace{3ex}
\\[2ex]
\\[2ex]

\section{Introduction}

Coupled flow systems containing a free-flow region and an adjacent porous medium
appear in a variety of environmental settings and industrial applications 
such as surface-subsurface flow interactions, industrial filtration and 
drying processes~\cite{Beaude_etal_19, Hanspal_etal_09, Reuter_etal_19}. Modeling flow and transport in such systems is a challenge since the behavior of the coupled model is very sensitive to the choice of interface conditions.

In the literature, there exist different model formulations to describe fluid 
flows in coupled systems, depending on the flow regime and the application of interest~\cite{Dawson_08, Discacciati_Quarteroni_09, 
Mosthaf_Baber_etal_11, Sochala_Ern_Piperno_09, Yang_etal_19}. In the most general case, the Navier--Stokes equations are applied to describe the free flow and the multiphase Darcy's law is used in the porous medium~\cite{Discacciati_Quarteroni_09, Mosthaf_Baber_etal_11, Yang_etal_19}. Interface conditions for the transport of chemical species and energy are needed as well, e.g.~\cite{AlazmiVafai, Mosthaf_Baber_etal_11}. For flows at low Reynolds numbers, the Stokes equations can be considered in the free-flow domain. 
In order to describe surface flow interactions with unsaturated porous media, the Stokes equations are coupled to the Richards equation~\cite{Rybak_etal_15}. When the porous medium is fully saturated, the single-phase Darcy law is considered in the subsurface and coupled to the Stokes equations in the free-flow domain~\cite{Angot_etal_17,Discacciati_Miglio_Quarteroni_02, Layton_Schieweck_Yotov_03, Bars_Worster_06,  OchoaTapia_Whitaker_95}.
The latter combination of models is the most widely used both for mathematical modeling, numerical analysis and development of efficient numerical algorithms, e.g.~\cite{Carraro_etal_13, Discacciati_Gerardo-Giorda_18, Discacciati_Miglio_Quarteroni_02, Jaeger_Mikelic_09, Kanschat_Riviere_10, Lacis_Bagheri_17,  Rybak_Magiera_14}. 
Other simplifications of the Navier--Stokes equations such as the shallow water equations, kinematic or diffusive waves can be coupled with different subsurface flow models~\cite{Dawson_08, Magiera_etal_15, Reuter_etal_19, Sochala_Ern_Piperno_09}. In this manuscript, we will be interested, however, in coupled Stokes--Darcy problems. 

Different sets of interface conditions are proposed in the 
literature to couple the Stokes and Darcy flow equations at the fluid--porous 
interface~\cite{Carraro_etal_15, Discacciati_Miglio_Quarteroni_02, Jaeger_etal_01, Lacis_Bagheri_17, Bars_Worster_06}. The most widely used coupling conditions are the conservation of mass, the balance of normal forces and the Beavers--Joseph or the 
Beavers--Joseph--Saffman condition for the tangential velocity component~\cite{Beavers_Joseph_67, Discacciati_Miglio_Quarteroni_02, Jones_73, Kanschat_Riviere_10, Layton_Schieweck_Yotov_03, Nield_09, Saffman_71}. However, these conditions are suitable for flows parallel to the fluid--porous interface only and thus not applicable for general filtration problems~\cite{Eggenweiler_Rybak_20}. 
Moreover, the Beavers--Joseph parameter needs to be fitted in order to take surface roughness and permeability of the interfacial zone into account. The exact location of the sharp fluid-porous interface is uncertain as well. There exist several recommendations for the interface location in the literature~\cite{Lacis_Bagheri_17, Rybak_etal_19}, however, only for circular solid grains leading to isotropic porous media.

Different alternatives to the classical interface conditions have been 
deve-loped 
within the last decades. Several modifications of the Beavers--Joseph--Saffman
condition have been derived rigorously using the theory of 
homogeni-zation 
and 
boundary layers for flows parallel to the fluid--porous
interface~\cite{Jaeger_Mikelic_96, Jaeger_Mikelic_09, Jaeger_etal_01}. Using the same techniques, coupling conditions for perpendicular flows to the interface are derived in~\cite{Carraro_etal_15}, however, they are limited to very specific boundary value problems. All
effective parameters appearing in coupling conditions derived by means of homogenization are computed numerically solving unit cell problems and boundary layer problems within a cut-off domain.

Alternative coupling concepts which are not restricted to the flow direction 
are proposed in~\cite{Angot_etal_17, Jackson_Rybak_etal_12}. However, these models contain several unknown parame-ters that still need to be determined and the specification of these effective coefficients is not an easy task. Another generalization of the classical interface conditions is developed in~\cite{Lacis_Bagheri_17, Lacis_etal_20}. 
In contrast to the previous concepts, the effective coefficients staying in the coupling conditions can be computed 
numeri-cally 
based on the pore-scale geometry of the coupled problem.
However, these conditions are validated only for parallel flows to the interface and very simple geometries, leaving the applicability of the proposed interface conditions to general flow problems an open question.

To summarize, existing alternative sets of interface conditions include un-known 
model parameters that need to be fitted~\cite{Angot_etal_17, Jackson_Rybak_etal_12}, are restricted to specific boundary value problems~\cite{Carraro_etal_15} and are neither validated for complicated geometri-cal 
configurations such as anisotropic media nor for arbitrary flow directions~\cite{Lacis_Bagheri_17, Lacis_etal_20}. Therefore, a need exists for interface conditions which are applicable to arbitrary fluid flows and anisotropic porous media and do not include any unknown coefficients. 

The objectives of this paper are (i) to rigorously derive a new set of coupling conditions which is valid for arbitrary flow directions to the  
interface, (ii) to compute all necessary effective model parameters based on the geometrical information of the coupled system, (iii) to validate the developed interface conditions numerically and (iv) to compare the proposed set of interface 
con-ditions 
with the classical one.

The paper is organized as follows. Geometrical configuration of the coupled flow problem, modeling assumptions, pore-scale resolved model and macroscale Stokes--Darcy model with the classical and new interface conditions are described in~\cref{sec:models}. The proposed set of interface conditions is derived rigorously in~\cref{sec:IC}. The developed conditions are validated numerically and compared with the classical coupling conditions for different flow problems in~\cref{sec:validation}. The conclusions follow in~\cref{sec:conclusions}.

\section{Coupled flow models}\label{sec:models}

In this section, we present the pore-scale model and the coupled macroscale Stokes--Darcy model with two different sets of interface conditions: the classical and the newly developed one. The pore-scale model with a fully resolved geometry serves as a basis for the homogenization approach in \cref{sec:IC} and for the validation of the derived interface conditions in \cref{sec:validation}.

\subsection{Geometrical configuration and assumptions}\label{sec:assumptions}

For the case of brevity, we present the derivation of the interface conditions for the two-dimensional case. For the pore-scale description we consider the flow region consisting of the free-flow domain and the pore-space of the porous medium, $\Omega^\varepsilon = \Omega_\FF \cup \Omega_\PM^\varepsilon \subset \RR^2$ (\cref{fig:geometry}a). The porous medium is assumed to be periodic, i.e. it is constructed by the periodic repetition of the scaled unit cell (\cref{tab:effective-properties}). The scale separation parameter is \mbox{$\varepsilon = \ell \slash L \ll 1$}, where $\ell$ is the characteristic pore size and $L$ is the length of the domain. For the macroscale description the coupled flow domain consists of the free-flow region and the adjacent porous-medium region, $\Omega = \Omega_\FF \cup \Omega_\PM \subset \RR^2$  (\cref{fig:geometry}c).

The interface $\Sigma$ is considered to be flat and simple, i.e. no storage of mass and momentum at the interface or transport of these properties along the interface is possible. 
For the sake of simplicity, we consider a horizontal interface $\Sigma$. Therefore, the unit normal vector on $\Sigma$ pointing outward from the porous-medium domain is $\vec n = \vec e_2$. 
Different porous-medium configurations (isotropic and  anisotropic) will be considered and different locations of the sharp fluid--porous interface $\Sigma$ (\cref{fig:geometry}b) for the macroscale model will be studied. 

\begin{figure}[htbp]
\centering
    \begin{minipage}{0.05\textwidth} 
        \vspace{-19ex}(a)  
   \end{minipage}
   \hspace{-4.ex}
   \begin{minipage}{0.4\textwidth} 
        \includegraphics[width=0.975\textwidth]{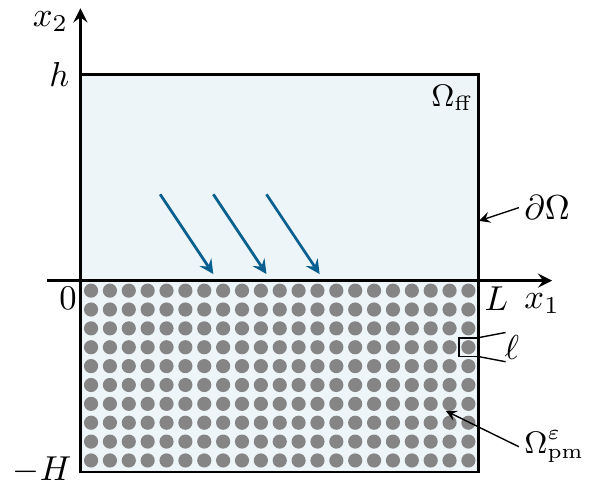} 
   \end{minipage}
   \hspace{-2.25ex}
       \begin{minipage}{0.05\textwidth} 
        \vspace{-6ex}(b)
   \end{minipage}
   \hspace{-2.ex}
   \begin{minipage}{0.15\textwidth} 
        \vspace{10ex}
        \includegraphics[width=0.925\textwidth]{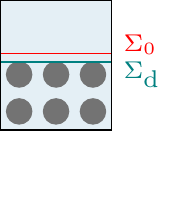} 
   \end{minipage}
   \hspace{-3.75ex}
   \begin{minipage}{0.05\textwidth} 
        \vspace{-19ex}(c)  
   \end{minipage}
   \hspace{-4.5ex}
   \begin{minipage}{0.4\textwidth} 
        \includegraphics[width=0.975\textwidth]{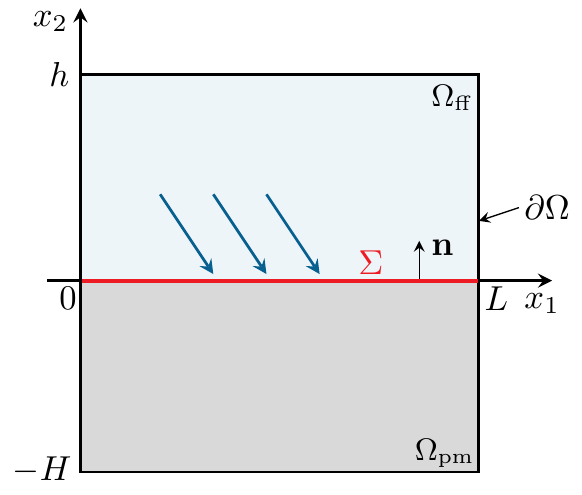} 
   \end{minipage}
   \caption[width=\linewidth]{Schematic geometrical configuration of the flow domain at the pore scale (a) and the macroscale (c). Two possible locations of the sharp interface (b).}
   \label{fig:geometry}
\end{figure}
We consider a single-phase flow of an incompressible fluid at low Reynolds numbers. 
The fluid is assumed to have constant viscosity and contain a single chemical species. This fluid occupies the free-flow domain and fully saturates the porous medium. The solid phase is supposed to be non-deformable and rigid leading to a constant porosity of the medium. The temperature of the fluid and the solid phase is assumed to be equal and constant, therefore no energy balance equation is needed. 

For the considered flow problems the fluid velocity in the free-flow region is of order $\mathcal O(1)$ and the Darcy velocity is much smaller, i.e. of order $\mathcal O(\varepsilon^2)$. These assumptions, however, break down close to the interface $\Sigma$ and deviations from the Darcy velocity and the free-flow velocity are expected there. 

For the theoretical derivation of new interface conditions, we consider the Stokes system with periodic boundary conditions on the lateral boundaries and an inflow condition on the upper boundary.
However, this requirement can be relaxed for numerical simulations (\cref{sec:validation-case-2,sec:validation-case-3}).

\subsection{Pore-scale model}
The incompressible fluid flow at low Reynolds numbers is described by the non-dimensional \textit{Stokes equations} in the flow region $\Omega^\varepsilon$ completed with the no-slip condition at the boundaries of the solid inclusions $\partial \Omega^{\varepsilon} \setminus \partial \Omega$ and appropriate boundary conditions at the external boundary $\partial \Omega$:
\vspace{-1ex}
\begin{subequations}
\begin{align}
- \Delta & \vec v^\varepsilon+ \nabla p^\varepsilon = \vec 0, \quad 
\operatorname{div} \vec{v}^{\varepsilon} = 0 \quad \textnormal{in } \Omega^{\varepsilon} , \quad 
\int_{\Omega_\FF} p^\varepsilon \ \text{d} \vec x = 0 ,
\\
\vec v^{\varepsilon} &= \vec 0 \quad \textnormal{on } \partial \Omega^{\varepsilon} \setminus \partial \Omega,    
\quad \vec{v}^{\varepsilon} = (v_1^\text{in}(x_1),0) \quad \text{on } \{x_2 = h\},   \label{eq:pore-scale-BC}
\\
v_2^{\varepsilon} &= \frac{\partial v_1^\varepsilon}{\partial x_2} = 0 \quad \text{on } \{x_2 = -H\}, \quad 
\{\vec{v}^{\varepsilon}, p^{\varepsilon} \}  \textnormal{ is $L$-periodic in $x_1$}.  \label{eq:pore-scale-3}
\end{align}
\label{eq:pore-scale}
\end{subequations}
Here $\vec v^\varepsilon=(v_1^\varepsilon, v_2^\varepsilon)$ and $p^\varepsilon$ are the fluid velocity and pressure. The inflow velocity $v_1^\text{in}(x_1)$  has to be chosen in such a way that the periodicity condition is  fulfilled. 

\subsection{Macroscale model}
Under the assumptions presented in~\cref{sec:assumptions} the macroscale model formulation  consists of the Stokes equations in the free-flow domain $\Omega_\FF$, Darcy's law in the porous-medium region $\Omega_\PM$ and an appropriate set of coupling conditions at the sharp fluid--porous interface $\Sigma$.

\subsubsection{Flow models}

The fluid flow in the free-flow domain $\Omega_\FF$ is governed by the \textit{Stokes equations} in dimensionless form and the following boundary conditions on the external boundary $\partial \Omega_\FF \setminus \Sigma$:
\begin{subequations}
\begin{align}
&-\Delta \vec v^\FF+ \nabla p^\FF = \vec 0, \quad \operatorname{div}\vec{v}^\FF = 0 \quad \textnormal{in } \Omega_\FF,  \quad \int_{\Omega_\FF} p^\FF \ \text{d} \vec x = 0, \label{eq:macro-ff-a}
\\[-1.ex]
&\vec{v}^\FF = (v_1^\text{in}(x_1),0) \quad \text{on } \{x_2 = h\}, 
\quad 
\{\vec{v}^\FF,  p^\FF \}  \textnormal{ is $L$-periodic in $x_1$}. \label{eq:macro-ff-b-periodic}
\end{align}
\label{eq:macro-ff}
\end{subequations}
Here, $\vec v^\FF = (v_1^\FF, v_2^\FF)$ and $p^\FF$ are the free-flow velocity and pressure and the inflow velocity $v_1^\text{in}$ is the same as in \cref{eq:pore-scale-BC}.

The flow in the porous-medium domain $\Omega_\PM$ is described by the
non-dimen-sional 
\textit{Darcy law} with the following boundary conditions on  $\partial \Omega_\PM \setminus \Sigma$:
\begin{subequations}
\begin{align}
    \vec v^\PM &= -\ten K^\varepsilon \nabla p^\PM, \quad \operatorname{div} \vec v^\PM = 0 \quad \textnormal{in } \Omega_{\PM}, \label{eq:macro-pm-a}
    \\
    v_2^\PM  = 0 & \quad  \text{on } \{x_2 = -H\} , 
    \quad 
    p^\PM   \textnormal{ is $L$-periodic is $x_1$}  \label{eq:macro-pm-b-periodic}, 
\end{align}
\label{eq:macro-pm}
\end{subequations}
where $\vec v^\PM = (v_1^\PM, v_2^\PM)$ and $p^\PM$ are the Darcy velocity and pressure and $\ten K^\varepsilon$ denotes the permeability tensor given by~\cref{eq:permeability}. 

\emph{Remark 2.1} (Boundary conditions).
Note that the normal velocity component on the upper and lower boundaries in conditions~\eqref{eq:pore-scale-BC},~\eqref{eq:pore-scale-3},~\eqref{eq:macro-ff-b-periodic} and \eqref{eq:macro-pm-b-periodic} 
is not necessarily zero, but should be chosen to satisfy the conservation of mass.

\subsubsection{Classical interface conditions}
For the Stokes--Darcy system \eqref{eq:macro-ff} and \eqref{eq:macro-pm} a variety of interface conditions has been proposed in the literature.  
However, independent of the flow regime, the most commonly used interface conditions are: \\ 
the \textit{conservation of mass} \\[-3.ex]
\begin{equation}\label{eq:IC-mass}
\vec v^\FF \vdot \vec n = \vec v^\PM \vdot \vec n \qquad \text{on} 
\;\Sigma,
\end{equation}
the {\it balance of normal forces}
\vspace*{-0.5ex}
\begin{equation}\label{eq:IC-momentum}
-\vec n \vdot \ten T (\vec v^\FF, p^\FF) \vdot \vec n = p^\PM 
\qquad \text{on} \; \Sigma,
\end{equation}
and the {\it Beavers--Joseph condition}~\cite{Beavers_Joseph_67,Jones_73} for the tangential velocity component
\begin{equation}\label{eq:IC-BJJ}
(\vec v^\FF - \vec v^\PM ) \vdot \vec \tau
-\alpha^{-1}\,\sqrt{\ten K^\varepsilon}
\vec \tau \vdot \nabla \vec v^\FF \vdot \vec n
= 0  \qquad \text{on} \; \Sigma.
\end{equation}
Here, $\alpha>0$ is the Beavers--Joseph parameter which is typically taken $\alpha = 1$ in the literature, $\vec n$ is the normal unit vector on $\Sigma$ pointing outward from the porous medium (\Cref{fig:geometry}c) 
and $\vec \tau$ is a tangential unit vector on $\Sigma$. The coupling conditions \cref{eq:IC-momentum,eq:IC-BJJ} are presented for the non-symmetric form of the stress~tensor 
$\ten T(\vec v^\FF, p^\FF) = \nabla \vec v^\FF - p^\FF \ten I$ in accordance to~\cref{eq:macro-ff-a}. However, also in symmetrized form, this set of coupling conditions is unsuitable for arbitrary flows to the porous medium~\cite{Eggenweiler_Rybak_20}.

\subsubsection{New interface conditions}

In \cref{sec:IC}, new \textit{interface conditions} 
which are valid for arbitrary flow directions to the interface are derived. We summarize them here for a complete macroscale model formulation
\begin{align}
    \vec v^\FF \vdot \vec n &= \vec v^\PM \vdot \vec n \quad \text{on } \Sigma, \label{eq:NEW-mass}
    \\
    p^\PM &=  p^\FF - \frac{\partial v_2^\FF}{\partial x_2} + N_s^{bl}  \frac{\partial v_1^\FF}{\partial x_2} \hspace{3.5mm}  \text{on } \Sigma, \label{eq:NEW-momentum}
    \\
    \vec v^\FF \vdot \vec \tau &= 
    - \varepsilon \vec N^{bl} \frac{\partial v_1^\FF}{\partial x_2} 
    \vdot \vec \tau
    + \varepsilon^2 \sum_{j=1}^2 \vec M^{j,bl} \frac{\partial p^\PM}{\partial x_j}
    \vdot \vec \tau  \quad \text{on } \Sigma. \label{eq:NEW-tangential}
\end{align}
Here, $\vec M^{j,bl}$ is the boundary layer constant given by~\cref{eq:BL-beta-exponentialDecay-2} which requires solving the boundary layer problem~\eqref{eq:BL-beta}, $\vec N^{bl}$ and $N_s^{bl}$ are the boundary layer constants given by~\cref{eq:BL-t-velocity-constant,eq:BL-t-pressure-constant} corresponding to problem~\eqref{eq:BL-t}.

With~\cref{eq:NEW-mass} we recovered the mass balance across the interface~\eqref{eq:IC-mass}. Further, the proposed coupling condition~\eqref{eq:NEW-momentum} is an extension of the balance of normal forces~\eqref{eq:IC-momentum} when considering the Stokes equations~\eqref{eq:macro-ff} with the non-symmetric stress tensor $\ten T(\vec v^\FF, p^\FF) = \nabla \vec v^\FF - p^\FF \ten I$ as it is done within this manuscript.
The interface condition~\eqref{eq:NEW-tangential} can be understood as a jump in tangential velocities
\begin{align}
    & ( \vec v^\FF - \vec v^{\PM})\vdot\vec\tau + 
     \varepsilon \vec N^{bl} \frac{\partial v_1^\FF}{\partial x_2} \vdot \vec \tau
    =
    \varepsilon^2\sum_{j=1}^2 \left( \vec M^{j,bl} + k_{1j} \right)\frac{\partial p^\PM}{\partial x_j} \vdot \vec \tau  \quad \text{on } \Sigma. \label{eq:new-IC-tangential-velocity-BJJ}
\end{align}
Note that  condition~\eqref{eq:new-IC-tangential-velocity-BJJ} is similar to the Beavers--Joseph condition~\eqref{eq:IC-BJJ} with the proportionality factor $-\varepsilon N_1^{bl} \sim \sqrt{\ten K^\varepsilon} \alpha^{-1}$. The difference between these conditions 
is that the right hand side in~\cref{eq:new-IC-tangential-velocity-BJJ} is not necessarily zero.

\emph{Remark 2.2}.
The proposed interface conditions~\eqref{eq:NEW-mass}--\eqref{eq:NEW-tangential} reduce to the interface conditions derived by J\"ager and Mikeli\'c under the same assumptions on the flow~\cite{Carraro_etal_13,Jaeger_Mikelic_96, Jaeger_Mikelic_00, Jaeger_etal_01}. We note that our interface conditions are more general, since no additional assumptions on the flow regime or flow direction are made.

\section{Derivation of the macroscale model formulation}\label{sec:IC}

To derive the macro-scale model formulation, we apply the theory of
homogeni-zation~\cite{Auriault_etal_09,Hornung_97}. 
Therefore, we study the behavior of the solutions to the pore-scale problem~\eqref{eq:pore-scale} when $\varepsilon \rightarrow 0$.
In this limit, the equations in the free-flow region remain valid, i.e. the Stokes equations describe the fluid flow in $\Omega_\FF$, whereas Darcy's law is obtained as the upscaled 
equation in the porous-medium domain $\Omega_\PM$.

\subsection{Darcy's law: Homogenization}

In this section, we summarize the derivation of Darcy's law and computation of permeability by means of homogenization with two-scale asymptotic expansions \cite{Auriault_etal_09,Hornung_97}. We assume that there exist asymptotic expansions of the velocity and pressure
\begin{subequations} \label{eq:asym-exp}
\begin{align}
    \vec v^\varepsilon (\vec x) &\approx \varepsilon^2 \vec v_0 (\vec x ,\vec y) + \varepsilon^3 \vec v_1(\vec x ,\vec y) + \mathcal{O}(\varepsilon^4), \\
     p^\varepsilon (\vec x) &\approx  p_0(\vec x ,\vec y) + \varepsilon  p_1(\vec x ,\vec y) + \mathcal{O}(\varepsilon^2), 
     \end{align}  
\end{subequations}
where $\vec y = \vec x/\varepsilon$ and  $\vec v_j, p_j$ are $1$-periodic in $\vec y$ for $j=0,1,2, \ldots$ 
Computing the derivatives $\nabla = \nabla_\vec{x} + \varepsilon^{-1} \nabla_\vec{y}$, substituting expansions~\eqref{eq:asym-exp} into the pore-scale problem~\eqref{eq:pore-scale} in $\Omega_\PM^\varepsilon$ and combining terms with the same degree of $\varepsilon$, we get~\cite[chap.~1.4]{Hornung_97}:
\begin{equation}
\label{eq:Darcy-homog}
    \vec v_0 = - \sum_{j=1}^2 \vec w^j\frac{\partial p_0}{\partial x_j}, \qquad
    p_0 = p^\PM,  \qquad  
    p_1 = - \sum_{j=1}^2 \pi^j\frac{\partial p_0}{\partial x_j}.
\end{equation}    

\begin{figure}[htbp]
  \centering
  \includegraphics[scale=0.85]{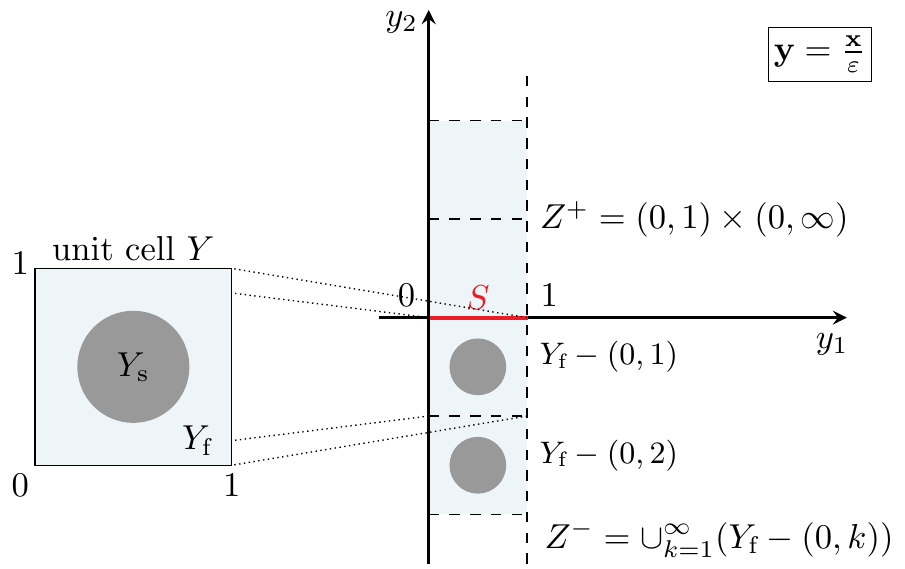}
  \caption{Unit cell and infinite boundary layer stripe $Z^{bl} = Z^+ \cup S \cup Z^-$.}
  \label{fig:stripe}
\end{figure}

We introduce the unit cell $Y=(0,1)^2$ and denote by $Y_\text{f}$ and $Y_\text{s}$ its fluid and solid part (\cref{fig:stripe}).
The functions $\vec{w}^{j} = (w_1^j,w_2^j)$ and $\pi^j$ are the solutions to the following cell problems for $j=1,2$:  
\begin{subequations}\label{eq:cell-prob}
\begin{align}
- \Delta_ {\vec{y}} \vec{w}^{j}  &+ \nabla_{\vec{y}} \pi^j = 
\vec{e}_j, \quad
\label{eq:cell-prob-momentum}
\operatorname{div}_{\vec{y}} \vec{w}^{j}  =0 \quad \text{in 
$Y_\text{f}$}, \quad \int_{Y_\text{f}} \pi^j \ \text{d} \vec{y} =0,
\\
\vec{w}^{j} &=  \vec{0} \quad \text{on $\partial Y_\text{f} \setminus \partial Y$},  \quad
\{ \vec{w}^{j}, \pi^j \} \text{ is 1-periodic} \text{ in } \vec{y}.
\end{align}
\end{subequations}

As a result of the homogenization process the non-dimensional divergence-free Darcy velocity is obtained
\begin{equation*}
\vec v^\PM(\vec x) = -\ten K^\varepsilon \nabla p^\PM (\vec x), \quad \operatorname{div} \vec v^\PM (\vec x) = 0, \quad \vec x \in \Omega_\PM,
\end{equation*}
where the permeability tensor $\ten K^\varepsilon$ is defined as follows
\begin{equation}\label{eq:permeability}
\ten K^\varepsilon = \varepsilon^2 \ten K = \varepsilon^2 (k_{ij})_{i,j=1,2} , \qquad k_{ij} = \int_{Y_\text{f}} 
w_{i}^j(\vec y) \ \text{d}  \vec{y}  .
\end{equation}

\subsection{Interface conditions: Homogenization and boundary layer theory}
In this section, we present the rigorous derivation of the new set of interface conditions~\eqref{eq:NEW-mass}--\eqref{eq:NEW-tangential} following the strategy proposed in~\cite{Carraro_etal_15,Jaeger_Mikelic_96}.
We define the space of test functions
\begin{align}\label{eq:Vperio}
V_\text{per}(\Omega^{\varepsilon}) = &\{ \vec{\phi} \in H^1({\Omega^{\varepsilon}})^2: \vec{\phi}= \vec 0 \text{ on } \partial \Omega^{\varepsilon} \setminus \partial \Omega, \ \vec{\phi}= \vec 0 \text{ on } \{ x_2 = h \}, \notag
\\
& \ \ \phi_2=0 \text{ on } \{ x_2 = -H \}, \ \vec{\phi} \textnormal{ is $L$-periodic in $x_1$} \}.
\end{align}
Our goal is to derive accurate approximations of the pore-scale velocity and pressure.
In order to achieve this, the following steps have to be done:
\\
\emph{Step 1.} 
    Construct the first approximations of the pore-scale velocity and pressure, $\vec v^{0,\varepsilon}_\text{approx}$ and $p^{0,\varepsilon}_\text{approx}$, and define the corresponding velocity and pressure error functions $\vec U^{0,\varepsilon}~=~\vec v^\varepsilon-~\vec v^{0,\varepsilon}_\text{approx}$ and $
    P^{0,\varepsilon} = p^\varepsilon - p^{0,\varepsilon}_\text{approx}$ (\cref{sec:first-approx}).
\\
\emph{Step 2.}  Write a variational formulation with respect to these error functions in order to detect the  
terms of low order with respect to $\varepsilon$.
\\
\emph{Step 3.}  
    Improve the first velocity and pressure approximations by adding appropriate boundary layer correctors and auxiliary functions (\cref{sec:first-approx,sec:continue-asym-exp,sec:trace-continuity,sec:lower-boundary,sec:compressibility}).
    Note that the corrections of the approximations  lead to changes in the corresponding error functions. We denote the subsequent approximations and error functions by rising index $n \in \mathbb{N}_0$, i.e.  $\vec v_\text{approx}^{n,\varepsilon}$, $p_\text{approx}^{n,\varepsilon}$, 
    $\vec U^{n,\varepsilon}$, $P^{n,\varepsilon}$.
\\
\emph{Step 4.}  
    Whenever the goal is achieved, prove error estimates of the model approximation (\cref{cor:4,cor:5}) which are necessary for the proof of two-scale convergence. However, the latter is out of the scope of this manuscript.

For the sake of clarity, we waive writing symbols $\text d \vec x$ and $\text{d} S$ at the end of  volume and boundary integrals, we  assume  $\vec \varphi \in V_\text{per}(\Omega^{\varepsilon})$, and write $q =q(\vec x)$ for $\vec x \in \Omega$
and $\nabla = \nabla_\vec{x}$, if not stated otherwise.

\subsubsection{First approximation of velocity and pressure}\label{sec:first-approx}
The variational formulation corresponding to the Stokes system~\eqref{eq:pore-scale} is given by
\begin{equation}\label{eq:weak-pore-scale}
    \int_{\Omega^\varepsilon} \nabla \vec v^\varepsilon \colon \nabla  \vec  \varphi - \int_{\Omega^\varepsilon} p^\varepsilon \operatorname{div} \vec \varphi = 0, \qquad \forall \vec\varphi \in V_\text{per}(\Omega^\varepsilon).
\end{equation}

The first step is to approximate the pore-scale solution in the free-flow region $\Omega_\FF$ and to eliminate the boundary conditions for $\vec v^\varepsilon$ at the upper boundary $\{x_2=h\}$. Therefore, the Stokes problem~\eqref{eq:macro-ff} is used in $\Omega_\FF$ and the variational formulation is 
\begin{align}\label{eq:weak-ff}
    \int_{\Omega_\FF} \nabla \vec v^\FF \colon \nabla  \vec  \varphi - \int_{\Omega_\FF} p^\FF \operatorname{div} \vec \varphi 
    &= -\int_{\Sigma} \underbrace{\left(
    \frac{\partial}{\partial x_2} 
    \vec v^\FF - \begin{bmatrix} 0 \\ p^\FF  \end{bmatrix}\right)}_{=\colon \vec F  } \vdot \vec \varphi.
\end{align}
Next, we introduce 
$\vec{w}^{j, \varepsilon}(\vec{x}) \!= \!\vec{w}^{j}(\vec{y})\!=\! \vec{w}^{j}(\vec{x}/\varepsilon), \ \pi^{j, \varepsilon}(\vec{x}) \!=\! \pi^{j}(\vec{x}/\varepsilon)$
for $\vec x \in \Omega^\varepsilon_\PM$, where $\{\vec w^j, \ \pi^j\}$ is given by~\cref{eq:cell-prob}, and extend $\vec{w}^{j, \varepsilon}$ by zero in $\Omega_\PM \setminus \Omega^\varepsilon_\PM$. 
Taking into account \cref{eq:cell-prob-momentum} and approximations  \eqref{eq:asym-exp}, \eqref{eq:Darcy-homog} of the pore-scale solutions in $\Omega_\PM^\varepsilon$, we obtain
\begin{align}\label{eq:weak-pm}
    \int_{\Omega_\PM^\varepsilon}& \nabla \left( -\varepsilon^2 \sum_{j=1}^2 \vec{w}^{j, \varepsilon} \frac{\partial p^\PM }{\partial x_j} \right) \colon \nabla  \vec  \varphi 
    - \int_{\Omega_\PM^\varepsilon} \left( p^\PM - \varepsilon \sum_{j=1}^2 \pi^{j,\varepsilon} \frac{\partial p^\PM}{\partial x_j} \right) \operatorname{div} \vec \varphi  \notag
    \\
    =& -  \!  \!  \int_{\Sigma}  \varepsilon \sum_{j=1}^2 \underbrace{\left( ( \varepsilon\nabla \vec{w}^{j, \varepsilon}  \!  -  \pi^{j,\varepsilon} \ten I) \frac{\partial p^\PM }{\partial x_j}  \vec e_2\right)}_{=\colon \vec B^j_\varepsilon} \vdot  \vec  \varphi 
    - \!  \!  \int_{\Sigma}   \varepsilon^2 \sum_{j=1}^2 \left( \left( \vec{w}^{j, \varepsilon}  \!  \otimes \nabla \frac{\partial p^\PM }{\partial x_j} \right)  \!  \vec e_2  \right) \vdot \vec  \varphi  \notag
    \\[-0.5ex]
    &-  \! \int_{\Sigma}p^\PM  \vec e_2 \vdot \vec  \varphi +\int_{\{x_2 = -H\}}  \varepsilon^2\sum_{j=1}^2  \left(\nabla \vec{w}^{j, \varepsilon} \frac{\partial p^\PM }{\partial x_j} \vec e_2 
    + \left( \vec{w}^{j, \varepsilon} \otimes  \nabla \frac{\partial p^\PM }{\partial x_j} \right) \vec e_2  \right) \vdot \vec  \varphi  \notag
    \\
    &+ \! \int_{\Omega_\PM^\varepsilon}   \sum_{j=1}^2  \bigg( \underbrace{\varepsilon^2\vec{w}^{j, \varepsilon} \Delta \frac{\partial p^\PM }{\partial x_j} + 2\varepsilon^2 \nabla \vec{w}^{j, \varepsilon} \nabla \frac{\partial p^\PM }{\partial x_j}  - \varepsilon \pi^{j,\varepsilon} \nabla \frac{\partial p^\PM}{\partial x_j}}_{=\colon - \vec A_\varepsilon^j} \bigg) \vdot \vec  \varphi\, \text{.}
\end{align}
To prove~\cref{cor:1}, we will need the following estimates for the terms in~\cref{eq:weak-pm}:
\vspace{-0.5ex}
\begin{align}
    \bigg| \int_{\Sigma} \varepsilon^2 \sum_{j=1}^2 \left(  \left( \vec{w}^{j, \varepsilon} \otimes \nabla \frac{\partial p^\PM }{\partial x_j} \right) \vec e_2 \right) \vdot \vec  \varphi \bigg| \leq& C \varepsilon^{5/2} \norm{\nabla \vec \varphi}_{L^2(\Omega_\PM^\varepsilon)^4},
    \label{eq:estimates-first-approx-1}
    \\[-1ex]
    \bigg| \int_{\{x_2 = -H\}}   \varepsilon^2 \sum_{j=1}^2   \left( \left( \vec{w}^{j, \varepsilon} \otimes \nabla \frac{\partial p^\PM }{\partial x_j} \right) \vec e_2 \right) \vdot \vec  \varphi \bigg| \leq&  C \varepsilon^{2} \norm{\nabla \vec \varphi}_{L^2(\Omega_\PM^\varepsilon)^4},
    \label{eq:estimates-first-approx-2}
    \\[-1ex]
    \bigg| \int_{\Omega_\PM^\varepsilon}   \sum_{j=1}^2 \vec A_\varepsilon^j \vdot \vec  \varphi \bigg| \leq& C \varepsilon^{2} \norm{\nabla \vec \varphi}_{L^2(\Omega_\PM^\varepsilon)^4}.
    \label{eq:estimates-first-approx-3}
\end{align}
We obtain these estimates using the Poincar\'{e} inequality and the results from \cite[Lemma 4.10]{Espedal_98}: 
For $\varphi \in H^1(\Omega^\varepsilon_\PM)$ with $\varphi= 0$ on $\partial \Omega^\varepsilon_\PM \setminus \partial \Omega_\PM$ we have
\begin{align}\label{eq:trace}
    \norm{ \varphi}_{L^2(\Omega_\PM^\varepsilon)} \leq & \ C \varepsilon \norm{\nabla \varphi}_{L^2(\Omega_\PM^\varepsilon)^2}, 
    \quad
    \norm{ \varphi}_{L^2(\Sigma)} \leq \ C \varepsilon^{1/2} \norm{\nabla \varphi}_{L^2(\Omega_\PM^\varepsilon)^2}.
\end{align}

Analyzing~\cref{eq:weak-pore-scale,eq:weak-ff,eq:weak-pm}, we get the first approximations of the pore-scale velocity $\vec v^\varepsilon$ and pressure $p^\varepsilon$:
\begin{subequations}\label{eq:approx-0}
\begin{align}
&\vec{v}^{0,\varepsilon}_\text{approx} = \mathcal{H}(x_2)\vec{v}^\FF + \mathcal{H}(-x_2) \  \bigg( - \varepsilon^2 \sum_{j=1}^2 \vec{w}^{j, \varepsilon} \frac{\partial p^\PM }{\partial x_j} \bigg),  \hspace{5mm}
\\[-0.5ex]
&p^{0,\varepsilon}_\text{approx} = \mathcal{H}(x_2)p^\FF + \mathcal{H}(-x_2) \bigg( p^\PM - \varepsilon \sum_{j=1}^2 \pi^{j,\varepsilon} \frac{\partial p^\PM}{\partial x_j} \bigg),
\end{align}
\end{subequations}
where $\mathcal{H}$ is the Heaviside function.

\subsubsection{Next order velocity approximation in the free flow}\label{sec:continue-asym-exp}

At this stage, we have $\vec v^{0,\varepsilon}_\text{approx} = \vec v^\FF \text{ in } \Omega_\FF$. It is shown in~\cite{Jaeger_Mikelic_96,Jaeger_etal_01} that this $\mathcal{O}(\varepsilon)$ approximation of the pore-scale velocity $\vec v^\varepsilon$ in the free-flow region is not sufficient for many coupled flow problems. 
Therefore, we  continue with the asymptotic expansions and consider the boundary layer problem \eqref{eq:BL-t} formulated in~\cite{Jaeger_Mikelic_96,Jaeger_etal_01}. This problem is defined on an infinite stripe $Z^{bl} = Z^{+} \cup S \cup Z^{-}$, where $Z^{+} = (0,1) \times (0,\infty)$, $S=(0,1) \times \{0\}$ and $Z^{-} = \cup_{k=1}^{\infty} (Y_\text{f} - (0,k))$. 
Here, $Y_\text{f} - (0,k)$ denotes the translation of the fluid part $Y_\text{f}$ in the negative $y_2$-direction for $k \in \mathbb{N}$. Thus, $Z^{bl}$ is the flow region within the stripe (\cref{fig:stripe}). We denote $\llbracket a\rrbracket _S := a(\cdot, +0)-a(\cdot,-0)$.

The boundary layer problem corresponding to the next order approximations is defined as follows
\begin{subequations}
\begin{align}
-\Delta_{\vec y} \vec t^{bl} + \nabla_{\vec y} s^{bl} &= \vec0 \quad \text{in } Z^{+}  \cup Z^{-}, \label{eq:BL-t-1}
\\
\operatorname{div}_{\vec y} \vec t^{bl} &= 0 \quad \text{in } Z^{+}  \cup Z^{-},  \label{eq:BL-t-2}
\\ 
\big\llbracket \vec t^{bl} \big\rrbracket_S &= \vec 0 \quad \text{on } S, \label{eq:BL-t-3}
\\
\big\llbracket ( \nabla_{\vec y} \vec t^{bl} - s^{bl} \ten{I} ) \vec{e}_2\big\rrbracket_S &=  \vec{e}_1 \hspace{1.5ex} \text{on } S, \label{eq:BL-t-4}
\\
\vec t^{bl} = \vec 0 \quad \text{on } \cup_{k=1}^{\infty}(\partial Y_s - (0,k)),& \quad \{\vec t^{bl}, s^{bl}\} \text{ is 1-periodic in $y_1$}. \label{eq:BL-t-5}
\end{align} \label{eq:BL-t}
\end{subequations}
Existence and uniqueness of $\vec t^{bl} \in L_{loc}^2(Z^{bl})^2$, $\nabla_{\vec y} \vec t^{bl} \in L^2(Z^{+} \cup Z^{-})^4$ and uniqueness up to a constant of $s^{bl} \in L_{loc}^2(Z^{bl})$ that satisfy \eqref{eq:BL-t} follows from the Lax--Milgram lemma~\cite[Proposition 3.22]{Jaeger_Mikelic_96}. After~\cite{Jaeger_Mikelic_96}, solutions $\vec t^{bl}$ and $s^{bl}$ stabilize exponentially towards boundary layer constants for $|y_2| \rightarrow \infty$. In order to define the 
boundary layer 
pressure $s^{bl}$ uniquely, we set
$\lim\limits_{\vec y \rightarrow -\infty}{s^{bl}(\vec y)} = 0$. 

As shown in~\cite{Jaeger_Mikelic_96}, there exist $\gamma \in(0, 1)$, $\vec{N}^{bl}$ and $N_{s}^{bl}$ such that
\begin{subequations}
\begin{align}
& e^{\gamma |y_2|}\nabla_\vec{y} \vec t^{bl} \in L^2(Z^+ \cup Z^-)^4, \ e^{\gamma |y_2|} \vec t^{bl} \in L^2(Z^-)^2, 
\ e^{\gamma |y_2|} s^{bl} \in L^2(Z^-) , \label{eq:BL-t-exponentialDecay}
\\
& \vec N^{bl} = (N_1^{bl}, 0) = \bigg(\int_S  t_1^{bl}(y_1, +0) \ \text{d} y_1, 0\bigg),  \label{eq:BL-t-velocity-constant}
\\
& N_{s}^{bl} = \int_0^1 s^{bl}(y_1, +0) \ \text{d} y_1, \label{eq:BL-t-pressure-constant}
\\
& |\vec t^{bl} - \vec{N}^{bl} | + |s^{bl} - N_s^{bl}| \leq Ce^{-\gamma y_2}, \quad y_2 > 0.
\end{align}
\end{subequations}
We introduce $\vec t^{bl,\varepsilon} (\vec x)= \vec t^{bl} (\vec y)$, $s^{bl,\varepsilon} (\vec x)= s^{bl} (\vec y)$ for $\vec x \in \Omega^\varepsilon$ and extend the boundary layer velocity $\vec t^{bl,\varepsilon}$ by zero in $\Omega \setminus \Omega^\varepsilon$. The following inequalities~\cite{Jaeger_Mikelic_96} hold
\begin{subequations}
\begin{align}
    \norm{\vec t^{bl,\varepsilon} - \mathcal{H}(x_2) \vec N^{bl} }_{L^2(\Omega)^2} &\leq C \varepsilon^{1/2}, \quad
    \norm{\nabla \vec t^{bl,\varepsilon}}_{L^2(\Omega_\FF \cup \Omega_\PM)^4} \leq C \varepsilon^{-1/2},
    \label{eq:BL-t-velocity-constant-estimates}
    \\
    \norm{s^{bl,\varepsilon} - \mathcal{H}(x_2) N_s^{bl} }_{L^2(\Omega^\varepsilon)} &\leq C \varepsilon^{1/2}. \label{eq:BL-t-pressure-constant-estimates}
\end{align}\label{eq:BL-t-properties} 
\end{subequations}

Taking into account the boundary layer velocity $\vec t^{bl,\varepsilon}$ and pressure $s^{bl,\varepsilon}$ for the next velocity and pressure approximations, we define the new error functions 
\begin{subequations}\label{eq:error-1}
\begin{align*}
    \vec U^{1,\varepsilon} 
    &= \vec U^{0,\varepsilon} \!+ \varepsilon \left(\vec t^{bl,\varepsilon} - \mathcal{H}(x_2)\vec N^{bl} \right) \frac{\partial v_1^\FF}{\partial x_2}\bigg|_\Sigma, 
    \\
    P^{1,\varepsilon} 
    &=  P^{0,\varepsilon} +  \left(s^{bl,\varepsilon} \!- \mathcal{H}(x_2) N_s^{bl} \right) \frac{\partial v_1^\FF}{\partial x_2}\bigg|_\Sigma .
\end{align*}
\end{subequations}
The variational formulation for $\vec U^{0,\varepsilon}, \ P^{0,\varepsilon} $ is obtained by combining~\cref{eq:weak-pore-scale,eq:weak-ff,eq:weak-pm}.
In order to derive the weak form corresponding to the new error functions $\vec U^{1,\varepsilon}, \ P^{1,\varepsilon}$, we focus here only on 
the newly added terms
\begin{align}\label{eq:weak-BL-t}
    \int_{\Omega^\varepsilon} \nabla &\left(  \varepsilon
    \left( \vec t^{bl,\varepsilon} - \mathcal{H}(x_2) \vec N^{bl} \right) \frac{\partial v_1^\FF}{\partial x_2}\bigg|_\Sigma \right) \colon \nabla \vec \varphi 
    + \int_{\Sigma}  N_s^{bl}  \frac{\partial v_1^\FF}{\partial x_2} \bigg|_\Sigma \vec e_2 \vdot \vec \varphi
    \notag
    \\
    &- \int_{\Omega^\varepsilon} \left(s^{bl,\varepsilon} - \mathcal{H}(x_2) N_s^{bl} \right) \frac{\partial v_1^\FF}{\partial x_2}\bigg|_\Sigma \operatorname{div} \vec \varphi
    \ +\int_{\Sigma}  \frac{\partial v_1^\FF}{\partial x_2}\bigg|_\Sigma \vec e_1 \vdot \vec \varphi 
    \notag
    \\
    =& \int_{\Omega^\varepsilon} \bigg(
    \underbrace{  \varepsilon  \left(\vec t^{bl,\varepsilon} - \mathcal{H}(x_2) \vec N^{bl} \right) \frac{\partial^2}{\partial x_1^2}\frac{\partial v_1^\FF}{\partial x_2}\bigg|_\Sigma}_{= \colon \vec A_\varepsilon^{11}} 
    +  \underbrace{ 
    \left(s^{bl,\varepsilon} - \mathcal{H}(x_2) N_s^{bl} \right) \nabla \frac{\partial v_1^\FF}{\partial x_2}\bigg|_\Sigma}_{= \colon \vec A_\varepsilon^{31}} \bigg) \vdot \vec \varphi \notag
    \\
    &+ 2\int_{\Omega^\varepsilon} \underbrace{ \varepsilon  \left( 
    \left( \vec t^{bl,\varepsilon} - \mathcal{H}(x_2)\vec N^{bl} \right) \otimes \frac{\partial}{\partial x_1}\frac{\partial v_1^\FF}{\partial x_2}\bigg|_\Sigma \vec e_1  \right) 
    }_{= \colon \ten A_\varepsilon^{21}} \colon \nabla \vec \varphi  + \text{ e.s.t.}
\end{align} 
Exponentially small terms are denoted by 'e.s.t.' throughout the manuscript.
The exponentially small terms in~\cref{eq:weak-BL-t} include the boundary layer corrector $\vec t^{bl,\varepsilon}$ and its gradient $\nabla \vec t^{bl,\varepsilon}$ appearing in the integral over~$\{x_2=-H\}$. 
The integral terms over the interface $\Sigma$ on the left hand side of~\cref{eq:weak-BL-t} will either vanish due to the derived interface conditions or cancel out with the corresponding parts of $\int_{\Sigma} \vec F  \vdot \vec \varphi$ from~\cref{eq:weak-ff}.
The following identity  
\begin{align}
    \diver \left( \vec t^{bl,\varepsilon} \otimes \nabla \frac{\partial v_1^\FF}{\partial x_2}\bigg|_\Sigma \right) = \vec t^{bl,\varepsilon}\Delta  \frac{\partial v_1^\FF}{\partial x_2}\bigg|_\Sigma + \nabla \vec t^{bl,\varepsilon} \nabla  \frac{\partial v_1^\FF}{\partial x_2}\bigg|_\Sigma \label{eq:trick}
\end{align}
is applied to obtain a higher estimation order for the terms on the right hand side of~\cref{eq:weak-BL-t} and is used later for similar calculations.
Taking into account the Poincar\'{e} inequality and estimates~\eqref{eq:trace} and \eqref{eq:BL-t-properties}, we obtain the following energy estimates
\begin{align}
    \bigg| \int_{\Omega_\FF} 
    \vec A_\varepsilon^{11}  \vdot \vec \varphi \bigg| &\leq C \varepsilon^{3/2} \norm{\vec \varphi}_{L^2(\Omega_\FF)^2},
    \quad
    \bigg| \int_{\Omega_\PM^\varepsilon} 
    \vec A_\varepsilon^{11}  \vdot \vec \varphi \bigg| 
    \leq C \varepsilon^{5/2} \norm{\nabla \vec \varphi}_{L^2(\Omega_\PM^\varepsilon)^4}, \label{eq:estimates-BL-t-1-2}
    \\
    \bigg| \int_{\Omega^\varepsilon} \ten A_\varepsilon^{21} \colon \nabla \vec \varphi \bigg| &\leq  C \varepsilon^{3/2} \norm{\nabla \vec \varphi}_{L^2(\Omega^\varepsilon)^4} , 
    \quad
    \bigg|  \int_{\Omega_\PM^\varepsilon} \vec A_\varepsilon^{31}  \vdot \vec \varphi \bigg| 
    \leq C \varepsilon^{3/2} \norm{\nabla \vec \varphi}_{L^2(\Omega_\PM^\varepsilon)^4}.  
    \label{eq:estimates-BL-t-3-4}
\end{align}
It remains to estimate $\int_{\Omega_\FF} 
\vec A_\varepsilon^{31} \vdot \vec \varphi$ on the right hand side of~\cref{eq:weak-BL-t}.  We follow the ideas from~\cite{Jaeger_Mikelic_96} and construct the auxiliary problem
\begin{align} \label{eq:V}
    \frac{\partial V}{\partial y_1} (\vec y)=s^{bl}(\vec y) - N_s^{bl}, \quad \vec y \in (0,1) \times (0,\infty),
    \quad
    V \text{ is $y_1$-periodic.}
\end{align}
From definition~\eqref{eq:BL-t-pressure-constant} of the boundary layer constant $N_s^{bl}$ it follows directly that
\begin{equation*}
    V(y_1,y_2) = \int_0^{y_1} s^{bl}(t,y_2) \ \text{d} t - N_s^{bl}y_1, \quad \vec y \in  (0,1) \times (0,\infty)
\end{equation*}
is a solution to~\cref{eq:V}.
Setting $V^{\varepsilon}(\vec x) = \varepsilon V(\vec y)$ for $\vec x \in \Omega_\FF$ and using~\cref{eq:BL-t-pressure-constant-estimates} we obtain
\vspace{-2ex}
\begin{align}
    \frac{\partial V^{\varepsilon}}{\partial x_1} = s^{bl,\varepsilon} - N_s^{bl} \quad \text{in } \Omega_\FF,
    \quad
    \norm{V^{\varepsilon}}_{L^2(\Omega_\FF)} \leq C  \varepsilon^{3/2}.
\end{align}
Taking into account the periodicity of function $V$ we get
\begin{align}\label{eq:estimate-trick}
    \bigg|  \int_{\Omega_\FF} \vec A_\varepsilon^{31}  \vdot \vec \varphi \bigg| 
    &=
    \bigg|  \int_{\Omega_\FF} 
    \left(s^{bl,\varepsilon} -  N_s^{bl} \right) \frac{\partial}{\partial x_1} \frac{\partial v_1^\FF}{\partial x_2}\bigg|_\Sigma \vec e_1 \vdot \vec  \varphi \bigg| \notag
    \\
    &=
    \bigg|  \int_{\Omega_\FF} V^\varepsilon \left( \varphi_1 \frac{\partial^2}{\partial x_1^2} \left(  \frac{\partial v_1^\FF}{\partial x_2}\bigg|_\Sigma \right) + \frac{\partial \varphi_1}{\partial x_1} \frac{\partial}{\partial x_1} \left(
    \frac{\partial v_1^\FF}{\partial x_2}\bigg|_\Sigma\right) \right) \bigg| 
    \notag
    \\[1ex]
    &\leq C \varepsilon^{3/2} \norm{\vec \varphi}_{H^1(\Omega_\FF)^2}  
    \leq C \varepsilon^{3/2} \norm{\nabla \vec \varphi}_{H^1(\Omega_\FF)^4}.
\end{align} 

Up to now, we do not have continuity of velocity trace across $\Sigma$. However, the principle of mass conservation across the interface should be fulfilled.
Hence, to obtain a physically consistent formulation we have to eliminate the trace jump on $\Sigma$.


\subsubsection{Velocity trace continuity}\label{sec:trace-continuity}

To establish the continuity of velocity trace across the interface $\Sigma$, we add boundary layer correctors which are solutions of the boundary layer problem for $j=1,2$ proposed  in~\cite{Carraro_etal_15}:
\begin{subequations}\label{eq:BL-beta}
\begin{align}
-\Delta_{\vec y} \vec \beta^{j,bl} + \nabla_{\vec y} \omega^{j,bl} &= \vec0  \quad \text{in } Z^{+}  \cup Z^{-}, \label{eq:BL-beta-1}
\\ \operatorname{div}_{\vec y} \vec \beta^{j,bl} &= 0 \quad \text{in } Z^{+}  \cup Z^{-}, \label{eq:BL-beta-2}
\\
\big\llbracket \vec \beta^{j,bl} \big\rrbracket_S &= k_{2j}\vec{e}_2 - \vec w^j \quad \text{on } S, \label{eq:BL-beta-3}
\\
\big\llbracket ( \nabla_{\vec y} \vec \beta^{j,bl} - \omega^{j,bl} \ten{I} ) \vec{e}_2\big\rrbracket_S &= - \left( \nabla_{\vec y} \vec w^{j} - \pi^j \right) \vec{e}_2 \quad \text{on } S, \label{eq:BL-beta-4}
\\
\vec \beta^{j,bl} = \vec 0 \quad \text{on } \cup_{k=1}^{\infty}(\partial Y_s - (0,k)),& \qquad \{\vec \beta^{j,bl}, \omega^{j,bl}\} \text{ is 1-periodic in $y_1$}. \label{eq:BL-beta-5}
\end{align}
\end{subequations}
\Cref{eq:BL-beta-3} establishes the continuity of normal velocity across the interface $\Sigma$ and \cref{eq:BL-beta-4} is necessary to remove the problematic term $\int_{\Sigma} \varepsilon \sum_{j=1}^2  \vec B^j_\varepsilon \vdot  \vec  \varphi$ in the weak formulation~\eqref{eq:weak-pm}.

Uniqueness of $\vec \beta^{j,bl} \in L_{loc}^2(Z^{bl})^2$, $\nabla_{\vec y} \vec \beta^{j,bl} \in L^2(Z^{+} \cup Z^{-})^4$ that satisfy \cref{eq:BL-beta} and uniqueness up to a constant of  $\omega^{j,bl} \in L_{loc}^2(Z^{bl})$ follows from the Lax--Milgram lemma. Using the results of~\cite{Carraro_etal_15,Jaeger_Mikelic_96} we know that~\eqref{eq:BL-beta} describes a boundary layer problem, 
therefore, velocity $\vec \beta^{j,bl}$ and pressure $\omega^{j,bl}$ stabilize exponentially towards boundary layer constants for $|y_2| \rightarrow \infty$.
Analogous to~\cref{sec:continue-asym-exp}, we set \mbox{$\lim\limits_{\vec y \rightarrow -\infty}{\omega^{j,bl}(\vec y)} = 0$}. There exist $\gamma \in (0,1)$, $\vec{M}^{j,bl}$ and  $M_{\omega}^{j,bl}$ such that
\begin{subequations}
\begin{align}
& e^{\gamma |y_2|}\nabla_\vec{y} \vec \beta^{j,bl} \! \! \in L^2(Z^+\! \cup \! Z^-)^4, e^{\gamma |y_2|} \vec \beta^{j,bl} \! \! \in L^2(Z^-)^2, \label{eq:BL-beta-exponentialDecay-1} 
e^{\gamma |y_2|} \omega^{j,bl} \in \! \! L^2(Z^-),
\\
& \vec{M}^{j,bl}= (M_1^{j,bl}, 0) = \bigg(\int_S \beta_1^{j, bl}(y_1, +0) \ \text{d} y_1, 0\bigg), \label{eq:BL-beta-exponentialDecay-2}
\\[-1ex]
& M_{\omega}^{j,bl} = \int_0^1 \omega^{j,bl}(y_1, +0) \ \text{d} y_1 ,
\\
& |\vec \beta^{j,bl} - \vec{M}^{j,bl} | + |\omega^{j,bl} - M_{\omega}^{j,bl}| \leq Ce^{-\gamma y_2}, \quad y_2 > 0.
\label{eq:BL-beta-exponentialDecay}
\end{align}
\end{subequations}

We introduce $\vec \beta^{j,bl, \varepsilon}(\vec{x}) = \vec \beta^{j,bl}(\vec{y})$,  $\omega^{j,bl, \varepsilon}(\vec{x})=\omega^{j,bl}(\vec{y})$ and extend the boundary layer velocity by setting $\vec \beta^{j,bl,\varepsilon}=\vec 0$ in $\Omega\setminus \Omega^{\varepsilon}$. Then, the following inequalities hold~\cite{Jaeger_Mikelic_96}:
\begin{subequations}\label{eq:BL-beta-estimates}
\begin{align}
    \| \vec \beta^{j,bl,\varepsilon}  -\mathcal{H}(x_2)  \vec{M}^{j,bl}  \|_{L^2(\Omega)^2} &\leq C\varepsilon^{1/2}, 
    \| \nabla \vec \beta^{j,bl,\varepsilon}  \|_{L^2(\Omega_\FF \cup \Omega_\PM)^4} \leq C\varepsilon^{-1/2}, 
    \\
    \| \omega^{j,bl,\varepsilon} -\mathcal{H}(x_2)  M_{\omega}^{j,bl}  \|_{L^2(\Omega^{\varepsilon})} &\leq C\varepsilon^{1/2}.
\end{align}
\end{subequations}

We use the boundary layer functions $\vec \beta^{j,bl,\varepsilon}, \ \omega^{j,bl,\varepsilon}$ and stabilizing constants $\vec M^{j,bl} , \ M_{\omega}^{j,bl} $ to improve the velocity and pressure approximations. This leads to the following error functions
\begin{subequations}\label{eq:error-2}
\begin{align*}
    \vec U^{2,\varepsilon} =
    &\vec U^{1,\varepsilon} -\varepsilon^2 \sum_{j=1}^2 \left( \vec{\beta}^{j,bl, \varepsilon} - \mathcal{H}(x_2) \vec M^{j,bl} \right) \frac{\partial p^\PM }{\partial x_j}\bigg|_\Sigma ,
    \\[-0.75ex]
    P^{2,\varepsilon} = 
    & P^{1,\varepsilon} - \varepsilon \sum_{j=1}^2 \left( \omega^{j,bl, \varepsilon} - \mathcal{H}(x_2)M_\omega^{j,bl} \right) \frac{\partial p^\PM }{\partial x_j}\bigg|_\Sigma.
\end{align*}
\end{subequations}

\begin{corollary}\label{cor:1}
The function $\vec U^{2,\varepsilon} \in H^1(\Omega^\varepsilon)^2$ and the following estimate holds
\begin{align}
\label{eq:cor:1}
    \bigg| & \int_{\Omega^\varepsilon} \nabla \vec U^{2,\varepsilon} \colon \nabla  \vec  \varphi - \int_{\Omega^\varepsilon} P^{2,\varepsilon} \operatorname{div} \vec \varphi 
    + \int\limits_{\{x_2 = -H\}}  \! \! \! \!  \varepsilon^2 \sum_{j=1}^2 \frac{\partial}{\partial x_2} {w}_1^{j, \varepsilon} \frac{\partial p^\PM }{\partial x_j} \varphi_1 \notag
    \\
    &
    - \! \int_{\Sigma} p^\PM   \varphi_2
    - \! \int_{\Sigma}  \left({
    \frac{\partial}{\partial x_2} \vec v^\FF} - \begin{bmatrix} 0 \\ p^\FF  \end{bmatrix} \right) 
    \vdot \vec \varphi 
    \ { 
    + \! \int_{\Sigma}  \frac{\partial v_1^\FF}{\partial x_2}\bigg|_\Sigma \vec e_1 \vdot \vec \varphi } \notag
    + \! \int_{\Sigma}  N_s^{bl}  \frac{\partial v_1^\FF}{\partial x_2} \bigg|_\Sigma \vec e_2 \vdot \vec \varphi
    \bigg|
    \\
    & \ \leq \ C\varepsilon^{3/2} \left(\norm{\nabla \vec \varphi}_{L^2(\Omega^\varepsilon)^4}
    +  \norm{\vec \varphi}_{H^1(\Omega_\FF)^2}\right), \qquad \forall \vec\varphi \in V_\text{per}(\Omega^\varepsilon).
\end{align}
\end{corollary}
\begin{proof}
By construction of approximation $\vec v^{2,\varepsilon}_\text{approx}$ we have $\vec U^{2,\varepsilon} \in H^1(\Omega^\varepsilon)^2$. Taking into account the fact that the boundary layer velocities $\vec t^{bl,\varepsilon}$, $\vec \beta^{j,bl,\varepsilon}$, their gradients and the pressures $s^{bl,\varepsilon}$, $\omega^{j,bl,\varepsilon}$ stabilize exponentially to zero for $x_2 \rightarrow -\infty$, we get
\begin{align}
    \int_{\Omega^\varepsilon}& \nabla \vec U^{2,\varepsilon} \colon \nabla  \vec  \varphi - \int_{\Omega^\varepsilon} P^{2,\varepsilon} \operatorname{div} \vec \varphi + \int\limits_{\{x_2 = -H\}}  \! \! \! \! \varepsilon^2 \sum_{j=1}^2 \frac{\partial}{\partial x_2} {w}_1^{j, \varepsilon} \frac{\partial p^\PM }{\partial x_j}  \varphi_1 \notag
    \\[-0.5ex] 
    & - \int_{\Sigma} p^\PM    \varphi_2 - \int_{\Sigma} \vec F \vdot  \vec  \varphi
    \ {
    +\int_{\Sigma}  \frac{\partial v_1^\FF}{\partial x_2}\bigg|_\Sigma \vec e_1 \vdot \vec \varphi }
    + \int_{\Sigma}  N_s^{bl}  \frac{\partial v_1^\FF}{\partial x_2} \bigg|_\Sigma \vec e_2 \vdot \vec \varphi \notag
     \\[-0.5ex]
    =&
    \int_{\Sigma}  \varepsilon^2 \sum_{j=1}^2 \left( \left(  \vec{w}^{j, \varepsilon} \! \otimes \nabla \frac{\partial p^\PM }{\partial x_j} \right) \vec e_2 \right) \vdot \vec  \varphi
    - \! \! \! \! \! \! \int\limits_{\{x_2 = -H\}}  \! \! \! \! \! \! \varepsilon^2 \sum_{j=1}^2  \left( \left(\vec{w}^{j, \varepsilon} \! \otimes \nabla \frac{\partial p^\PM }{\partial x_j} \right) \vec e_2 \right) \vdot \vec  \varphi
    \notag
    \\[-0.5ex]
    & + \int_{\Omega_\PM^\varepsilon}   \sum_{j=1}^2 \vec A_\varepsilon^j \vdot \vec  \varphi
    +  \int_{\Omega^\varepsilon_\PM} 
    \vec A_\varepsilon^{11} \vdot \vec \varphi
     +  \int_{\Omega_\FF} 
    \vec A_\varepsilon^{11} \vdot \vec \varphi
    + 2 \int_{\Omega^\varepsilon} \ten A_\varepsilon^{21} \colon \nabla \vec \varphi \notag
    \\[-0.5ex]
    &+ \int_{\Omega^\varepsilon_\PM} \vec A_\varepsilon^{31} \vdot \vec \varphi
    + \int_{\Omega_\FF} \vec A_\varepsilon^{31} \vdot \vec \varphi + \int_\Sigma \varepsilon \sum_{j=1}^2  M_\omega^{j,bl} \frac{\partial p^\PM}{\partial x_j}\bigg|_\Sigma \varphi_2 \notag
    \\[-1ex]
    &- 2 \int_{\Omega^\varepsilon} \sum_{j=1}^2  
    \underbrace{\varepsilon^2 \left(\left( \vec \beta^{j,bl,\varepsilon} - \mathcal{H}(x_2) \vec M^{j,bl}\right)  \otimes  \frac{\partial}{\partial x_1}   \frac{\partial p^\PM}{\partial x_j}\bigg|_\Sigma  \vec e_1 \right)}_{=\colon \ten A_\varepsilon^{j,12}} \colon \nabla  \vec  \varphi
    \notag
    \\[-1ex]
    &
    - \int_{\Omega^\varepsilon} \sum_{j=1}^2 \underbrace{\varepsilon^2 \left( \vec \beta^{j,bl,\varepsilon} - \mathcal{H}(x_2)  \vec M^{j,bl} \right) \frac{\partial^2}{\partial x_1^2} \frac{\partial p^\PM}{\partial x_j}\bigg|_\Sigma}_{= \colon \vec A_\varepsilon^{j,22}} \vdot \vec \varphi 
    \notag
    \\[-1ex]
    &- \int_{\Omega^\varepsilon} \sum_{j=1}^2
    \underbrace{\varepsilon \left(\omega^{j,bl,\varepsilon} -\mathcal{H}(x_2)  M_\omega^{j,bl} \right) \frac{\partial}{\partial x_1} \frac{\partial p^\PM}{\partial x_j}\bigg|_\Sigma \vec e_1}_{= \colon\vec A_\varepsilon^{j,32}}
    \vdot \vec \varphi + \text{ e.s.t.}
    \label{eq:proof-1}
\end{align}
The exponentially small terms in~\cref{eq:proof-1} include the boundary layer velocities and their gradients appearing in integrals over the lower boundary~$\{x_2=-H\}$. 

Using the Poincar\'{e} inequality and inequalities~\eqref{eq:trace} and ~\eqref{eq:BL-beta-estimates} we can estimate the following integral terms in~\cref{eq:proof-1} as
\begin{align*}
    & \bigg|\! \int\limits_\Sigma \!
    \sum_{j=1}^2  M_\omega^{j,bl} \frac{\partial p^\PM}{\partial x_j}\bigg|_\Sigma \! \varphi_2
    \bigg| \leq C 
    \varepsilon^{1/2} 
    \norm{\nabla \vec \varphi}_{L^2(\Omega^\varepsilon)^4}, \; 
    \bigg|\!\int\limits_{\Omega^\varepsilon}\! \sum_{j=1}^2 \ten A_\varepsilon^{j,12} \!\colon \!\nabla  \vec  \varphi \bigg| \!\leq C \varepsilon^{5/2} \norm{\nabla \vec \varphi}_{L^2(\Omega^\varepsilon)^4},
    \\[-0.5ex]
    &\bigg|\int_{\Omega^\varepsilon_\PM} \sum_{j=1}^2 \vec A_\varepsilon^{j,22} \vdot
    \vec  \varphi \bigg| \leq C \varepsilon^{3} \norm{\nabla \vec \varphi}_{L^2(\Omega_\PM^\varepsilon)^4}, \;\;
    \bigg|\int_{\Omega^\varepsilon_\PM} \sum_{j=1}^2 \vec A_\varepsilon^{j,32} \vdot \vec \varphi \bigg| \leq C \varepsilon^{2} \norm{\nabla \vec \varphi}_{L^2(\Omega_\PM^\varepsilon)^4},
    \\[-0.5ex]
    &\bigg|\int_{\Omega_\FF} \sum_{j=1}^2 \vec  A_\varepsilon^{j,22} \vdot \vec \varphi \bigg| \leq C \varepsilon^{5/2} \norm{\nabla \vec \varphi}_{L^2(\Omega_\FF)^4}, \; \;
    \bigg|\int_{\Omega_\FF} \; \; \sum_{j=1}^2 \vec A_\varepsilon^{j,32} \vdot \vec \varphi \bigg| \leq C \varepsilon^{5/2} \norm{\vec \varphi}_{H^1(\Omega_\FF)^2}. 
\end{align*}

The last estimate is 
obtained using the same idea as for estimate~\eqref{eq:estimate-trick}. 
Taking into account~\cref{eq:estimates-first-approx-1,eq:estimates-first-approx-2,eq:estimates-first-approx-3,eq:estimates-BL-t-1-2,eq:estimates-BL-t-3-4,eq:estimate-trick,eq:BL-beta-estimates} we complete the proof.
\end{proof}

The goal is to construct a velocity approximation $\vec v^{n,\varepsilon}_\text{approx}$ such that for some $n \in \mathbb{N}_0$ the corresponding error function $\vec U^{n,\varepsilon} \in V_\text{per}(\Omega^\varepsilon)$  can be used as a test function in~\cref{eq:proof-1}. So far, $\vec U^{2,\varepsilon}$ does not fulfill the boundary conditions on $\{x_2=-H\}$. Hence, the next step is to adjust the values on the lower boundary.

\subsubsection{Velocity correction on the lower boundary}\label{sec:lower-boundary}
The velocity approximation $\vec{v}^{2,\varepsilon}_\text{approx}$ does not satisfy the boundary conditions~\eqref{eq:pore-scale-3} on the lower boundary $\{x_2=-H\}$. Thus, another boundary layer corrector is needed.
We analyze the boundary values to detect the problematic terms
\begin{subequations}
\begin{align}
 U_2^{2,\varepsilon} (x_1, -H)
 =&  \varepsilon^2 \sum_{j=1}^2 (w_2^{j}(y_1, 0) - k_{2j}) \frac{\partial p^\PM}{\partial x_j} (x_1, -H)  + \textnormal{ e.s.t.} \label{eq:value-at-lower-boundary}
 \\
 \frac{\partial U_1^{2,\varepsilon}}{\partial x_2} (x_1, -H)
 =&  \varepsilon \sum_{j=1}^2  \frac{\partial}{\partial y_2}   w_1^{j}(y_1,0) \frac{\partial p^\PM}{\partial x_j}(x_1, -H) \notag 
 \\
 & + \varepsilon^2 \sum_{j=1}^2   w_1^{j}(y_1,0)  \frac{\partial}{\partial x_2} \frac{\partial p^\PM}{\partial x_j}(x_1, -H) + \textnormal{ e.s.t.} \label{eq:lower-2}
\end{align}
\end{subequations}

\Cref{eq:value-at-lower-boundary} is obtained by adding $-\varepsilon^2 \sum_{j=1}^2 k_{2j} \partial p^\PM / \partial x_j (x_1,-H) = 0$ to the second component of the error function $\vec U^{2,\varepsilon}$. Moreover, the periodicity of the unit cell velocity $\vec w^{j}$ is used, $\vec w^j (y_1, -H \varepsilon^{-1}) = \vec w^j (y_1, 0)$.
Since the right hand sides in~\cref{eq:value-at-lower-boundary,eq:lower-2} are up to the factor $\varepsilon^2$ the same as those in~\cite[section~4.3]{Carraro_etal_15}, we can use the same boundary layer problem to correct the outer boundary effects
\begin{subequations}
\begin{align}
-\Delta_{\vec y} \vec q^{j,bl} + \nabla_{\vec y} z^{j,bl} &= \vec0 \quad \text{in } Z^{-}, \label{eq:BL-q-1}
\\
\operatorname{div}_{\vec y} \vec q^{j,bl} &= 0 \quad \text{in } Z^{-},  \label{eq:BL-q-2}
\\ 
q_2^{j,bl}&= k_{2j} - w_2^j \hspace{3ex} \text{on } S, \label{eq:BL-q-3}
\\
\frac{\partial q_1^{j,bl}}{\partial y_2}  &= - \frac{\partial w_1^{j}}{\partial y_2} \hspace{5.5ex} \text{on } S, \label{eq:BL-q-4}
\\
\vec q^{j,bl} = \vec 0 \  \text{ on } \cup_{k=1}^{\infty}(\partial Y_s - (0,k))&, \quad \{\vec q^{j,bl}, z^{j,bl}\} \text{ is 1-periodic in $y_1$}. \label{eq:BL-q-5}
\end{align}\label{eq:BL-q}
\end{subequations}
After~\cite{Carraro_etal_15,Jaeger_Mikelic_96} there exists a unique solution $\vec q^{j,bl} \in H^1(Z^-)^2$, smooth in $Z^-$ and a constant $\gamma \in(0,1)$ such that $e^{\gamma |y_2|}\vec q^{j,bl} \in L^2(Z^-)^2$. The pressure $z^{j,bl}$ is unique up to a constant. The determination of the constant yields $e^{\gamma |y_2|} z^{j,bl} \in L^2(Z^-)$. 

Taking into account the boundary layer problem~\eqref{eq:BL-q} the error functions read
\begin{subequations}\label{eq:error-3}
\begin{align*}
 \vec U^{3,\varepsilon} &= \vec U^{2,\varepsilon} + \varepsilon^2 \sum_{j=1}^2  \frac{\partial p^\PM }{\partial x_j}(x_1, -H) \vec q^{j,bl}\left(\frac{x_1}{\varepsilon}, -\frac{x_2+H}{\varepsilon} \right), 
 \\
 P^{3,\varepsilon} &=  P^{2,\varepsilon} + \varepsilon \sum_{j=1}^2 \frac{\partial p^\PM }{\partial x_j}(x_1, -H) z^{j,bl}\left(\frac{x_1}{\varepsilon}, -\frac{x_2+H}{\varepsilon} \right).
\end{align*}
\end{subequations}

\begin{corollary}\label{cor:2}
The error function $\vec U^{3,\varepsilon} \in V_\text{per}(\Omega^\varepsilon)$ and it holds
\begin{align}
    & \bigg|  \int_{\Omega^\varepsilon} \nabla \vec U^{3,\varepsilon} \colon \nabla  \vec  \varphi - \int_{\Omega^\varepsilon} P^{3,\varepsilon} \operatorname{div} \vec \varphi 
    - \int_{\Sigma} p^\PM  \vec e_2 \vdot \vec  \varphi
    -\int_{\Sigma} \left(
    \frac{\partial v^\FF_2}{\partial x_2} - p^\FF  \right) \varphi_2
    \notag
    \\
    &+ \int_{\Sigma}  N_s^{bl}  \frac{\partial v_1^\FF}{\partial x_2} \bigg|_\Sigma \vec e_2 \vdot \vec \varphi \bigg|
    \leq C\varepsilon^{3/2} \left( \norm{\nabla \vec \varphi}_{L^2(\Omega^\varepsilon)^4}
    + \norm{\vec \varphi}_{H^1(\Omega_\FF)^2} \right), \; \forall \vec\varphi \in V_\text{per}(\Omega^\varepsilon).
    \label{eq:cor:2}
\end{align}
\end{corollary}

\begin{proof}
From~\cref{cor:1} we have $\vec U^{2,\varepsilon} \in H^1(\Omega^\varepsilon)^2$. Taking into account that $\vec q^{j,bl} \in H^1(Z^-)^2$, we get $\vec U^{3,\varepsilon} \in H^1(\Omega^\varepsilon)^2$. Correction of the velocity on the lower boundary yields $\vec U^{3,\varepsilon} \in V_\text{per}(\Omega^\varepsilon)$. As we already proved estimate~\eqref{eq:cor:1}, we only focus on the newly added terms. The variational formulation corresponding to the newly added terms to $\vec U^{2,\varepsilon},\ P^{2,\varepsilon}$, caused by boundary layer problem~\eqref{eq:BL-q}, reads
\begin{align}
    \int_{\Omega^\varepsilon}& \nabla \left( \varepsilon^2 \sum_{j=1}^2  \frac{\partial p^\PM }{\partial x_j}(x_1, -H) \vec q^{j,bl}\left(\frac{x_1}{\varepsilon}, -\frac{x_2+H}{\varepsilon} \right)
    \right) \colon \nabla  \vec  \varphi \notag
    \\
    &- \int_{\Omega^\varepsilon}   \varepsilon \sum_{j=1}^2 \frac{\partial p^\PM }{\partial x_j}(x_1, -H) z^{j,bl}\left(\frac{x_1}{\varepsilon}, -\frac{x_2+H}{\varepsilon} 
    \right) \operatorname{div} \vec \varphi \notag
    \\ 
    &-\int_{\{x_2 = -H\}}   \varepsilon^2 \sum_{j=1}^2  \frac{\partial p^\PM }{\partial x_j}(x_1, -H)  \frac{\partial w_1^{j,\varepsilon}}{\partial x_2}(x_1,0) 
    \varphi_1 
    \notag
    \\
    =
    &- \int_{\Omega^\varepsilon_\PM} \underbrace{ \varepsilon^2  \sum_{j=1}^2  \frac{\partial^2}{\partial x_1^2} \frac{\partial p^\PM }{\partial x_j}(x_1, -H) \vec q^{j,bl}\left(\frac{x_1}{\varepsilon}, -\frac{x_2+H}{\varepsilon} 
    \right)}_{=: \vec A_\varepsilon^{j,13}} \vdot  \vec  \varphi \notag
     \\[-1ex]
    &- 2 \int_{\Omega^\varepsilon_\PM} \underbrace{ \varepsilon^2 \sum_{j=1}^2  \nabla \frac{\partial p^\PM }{\partial x_j}(x_1, -H) \nabla \vec q^{j,bl}\left(\frac{x_1}{\varepsilon}, -\frac{x_2+H}{\varepsilon} 
    \right)}_{=: \vec A_\varepsilon^{j,23}} \vdot  \vec  \varphi \notag
    \\[-1ex]
    &+ \int_{\Omega^\varepsilon_\PM} \underbrace{ \varepsilon \sum_{j=1}^2 \nabla \frac{\partial p^\PM }{\partial x_j}(x_1, -H) z^{j,bl}\left(\frac{x_1}{\varepsilon}, -\frac{x_2+H}{\varepsilon}
    \right)}_{=: \vec A_\varepsilon^{j,33}} \vdot \vec \varphi + \text{ e.s.t.}
    \label{eq:proof-2}
\end{align} 
The integral over $\{x_2=-H\}$ on the left hand side in~\cref{eq:proof-2} cancels with the corresponding integral on the left hand side in~\cref{eq:cor:1} taking into account the periodicity of $\vec w^{j,\varepsilon}$.
Note that the boundary layer velocity and pressure jumps across the lower boundary are zero, because the functions $\vec q^{j,bl}$ and $z^{j,bl}$ are continuous. The integrals over $\Omega_\FF$ are exponentially small due to the fact that $\vec q^{j,bl},  \ z^{j,bl} \rightarrow 0$ for $x_2 \rightarrow - \infty$, therefore, they do not appear in~\cref{eq:proof-2}. 

Using the Poincar\'{e} inequality and inequality~\eqref{eq:estimate-trick} we obtain the following estimates for the integrals on the right hand side of~\cref{eq:proof-2}:
\begin{align}
    &\bigg| \! \int\limits_{\Omega^\varepsilon_\PM} \!  \!  \sum_{j=1}^2 \vec 
    A^{j,13}_\varepsilon  \vdot \vec  \varphi   \bigg| 
    \leq C \varepsilon^3 \norm{\nabla \vec \varphi}_{L^2(\Omega^\varepsilon_\PM)^4},
    \ 
    \bigg| \! \int\limits_{\Omega^\varepsilon_\PM} \sum_{j=1}^2 \! \!  \vec A^{j,23}_\varepsilon  \vdot \vec  \varphi \bigg| 
    \leq C \varepsilon^2 \norm{\nabla \vec \varphi}_{L^2(\Omega^\varepsilon_\PM)^4}, \label{eq:estimates-q-1}
    \\[-0.5ex]
    &\bigg| \! \int\limits_{\Omega^\varepsilon_\PM} \!   \sum_{j=1}^2 \vec A_\varepsilon^{j,33} \vdot \vec \varphi \bigg|
    \leq C \varepsilon^2 \norm{\nabla \vec \varphi}_{L^2(\Omega^\varepsilon_\PM)^4} .\label{eq:estimates-q-2}
\end{align}
Taking into account~\cref{eq:estimates-q-1,eq:estimates-q-2,eq:proof-1,eq:proof-2}, we obtain estimate~\eqref{eq:cor:2}.
\end{proof}
To eliminate the integrals over $\Sigma$ on the left hand side of~\cref{eq:cor:2} we set 
\begin{align}
p^\PM = p^\FF - \frac{\partial v_2^\FF}{\partial x_2} + N_s^{bl} \frac{\partial v_1^\FF}{\partial x_2} \quad \text{on } \Sigma. \label{eq:pressure-IC-derivation}
\end{align}
Therewith, we can rewrite inequality~\eqref{eq:cor:2} as follows
\begin{align}
    \bigg| \int_{\Omega^\varepsilon}& \nabla \vec U^{3,\varepsilon} \colon \nabla  \vec  \varphi - \int_{\Omega^\varepsilon} P^{3,\varepsilon} \operatorname{div} \vec \varphi 
    \bigg|
    \leq C\varepsilon^{3/2} \left( \norm{\nabla \vec \varphi}_{L^2(\Omega^\varepsilon)^4}
    + \norm{\vec \varphi}_{H^1(\Omega_\FF)^2} \right). \label{eq:cor:3-new}
\end{align}
It remains to estimate the divergence of the velocity error function $\vec U^{3,\varepsilon}$. 

\subsubsection{Correction of compressibility effects}
\label{sec:compressibility}
Consider the divergence of the velocity error function 
\vspace{-1.5ex}
\begin{align} \label{eq:divergence}
    \operatorname{div} \vec U^{3,\varepsilon}
    &=  \mathcal{H}(-x_2) \ \varepsilon^2 \sum_{j=1}^2 \vec{w}^{j, \varepsilon} \vdot \nabla \frac{\partial p^\PM }{\partial x_j} 
    + \varepsilon 
    \left(t_1^{bl,\varepsilon}- \mathcal{H}(x_2)N_1^{bl} \right)  \frac{\partial}{\partial x_1} \frac{\partial v_1^\FF}{\partial x_2}\bigg|_\Sigma \notag
    \\[-0.5ex]
    & \quad -\varepsilon^2 \sum_{j=1}^2 \left( {\beta_1}^{j,bl, \varepsilon} - \mathcal{H}(x_2) M_1^{j,bl} \right) \frac{\partial}{\partial x_1} \frac{\partial p^\PM }{\partial x_j}\bigg|_\Sigma \notag
    \\[-0.5ex]
    & \quad + \varepsilon^2 \sum_{j=1}^2  \frac{\partial}{\partial x_1} \frac{\partial p^\PM }{\partial x_j}(x_1, -H) q_1^{j,bl}\left(\frac{x_1}{\varepsilon}, -\frac{x_2+H}{\varepsilon} \right).
\end{align}
Taking into account~\cref{eq:BL-t-velocity-constant-estimates}, we have $\norm{\operatorname{div} \vec U^{3,\varepsilon}} \leq C \varepsilon^{3/2}$. However, as already mentioned in~\cref{sec:continue-asym-exp}, the velocity approximation in the free-flow region should be at least of order $\mathcal{O}(\varepsilon^{3/2})$. Thus, the divergence of the velocity error needs to be at least of order $\mathcal{O}(\varepsilon^{5/2})$. 
We consider each contribution to the divergence in a separate step \emph{a)~--~d)} and correct it using boundary layers and auxiliary functions if necessary. 

\paragraph{a) Compressibility effects coming from the term with $\vec w^{j,\varepsilon}$}
To correct the compressibility effects coming from the cell problems, we consider the 
aux-iliary 
function $\vec \gamma^{j,i}$ for $i,j=1,2$ satisfying~\cite[section~4.4]{Carraro_etal_15}:
\vspace{-2ex}
\begin{subequations}
\begin{align}
\operatorname{div}_{\vec y} \vec \gamma^{j,i}  &= w_i^j - \frac{k_{ij}}{|Y_\text{f}|} \quad \text{in } Y_\text{f}, \label{eq:aux-func-1 }
\\
 \vec \gamma^{j,i} = \vec 0 \quad \text{on } \partial Y_\text{f} \setminus \partial Y,& \quad \vec \gamma^{j,i} \text{ is 1-periodic in $\vec y$}. \label{eq:aux-func-2}
\end{align}\label{eq:aux-func}
\end{subequations}
After~\cite{Carraro_etal_15,Jaeger_Mikelic_96}, there exists at least one solution $\vec \gamma^{j,i} \in H^1(Y_\text{f})^2 \cap C^\infty_\text{loc}(\cup_{k=1}^{\infty}( Y_\text{f} - (0,k))^2) $ to problem~\eqref{eq:aux-func}. We introduce $\vec \gamma^{j,i, \varepsilon}(\vec x) = \varepsilon \vec \gamma^{j,i} (\vec y)$ for $\vec x \in \Omega_\PM^\varepsilon$ and extend it by zero in $\Omega_\PM \setminus \Omega_\PM^\varepsilon$. 
Since $\vec \gamma^{j,i,\varepsilon}$ is defined in the porous-medium domain only, we have to correct its values on the interface $\Sigma$. In order to do this, we use the following boundary layer problem~\cite{Carraro_etal_15,Jaeger_Mikelic_96}:
\vspace{-1ex}
\begin{subequations}
\begin{align}
-\Delta_{\vec y} \vec \gamma^{j,i,bl} + \nabla_{\vec y} \pi^{j,i,bl} &= \vec0 \quad \text{in } Z^{+}  \cup Z^{-}, \label{eq:BL3a-1}
\\
\operatorname{div}_{\vec y} \vec \gamma^{j,i,bl } &= 0 \quad \text{in } Z^{+}  \cup Z^{-}, \label{eq:BL3a-2}
\\ 
\big\llbracket \vec \gamma^{j,i,bl} \big\rrbracket_S &= \vec \gamma^{j,i} \hspace{9.5ex} \text{on } S, \label{eq:BL3a-3}
\\
\big\llbracket ( \nabla_{\vec y} \vec \gamma^{j,i,bl} - \pi^{j,i,bl} \ten{I} ) \vec{e}_2\big\rrbracket_S &= - \nabla_{\vec y } \vec \gamma^{j,i} \vec e_2  \quad \text{on } S, \label{eq:BL3a-4}
\\
\vec \gamma^{j,i,bl} = \vec 0 \quad \text{on } \cup_{k=1}^{\infty}(\partial Y_s - (0,k)),& \quad \{\vec \gamma^{j,i,bl}, \pi^{j,i,bl}\} \text{ is 1-periodic in $y_1$}. \label{eq:BL3a-5}
\end{align}  \label{eq:BL3}
\end{subequations}
It is proven in~\cite{Jaeger_Mikelic_96} that there exist a constant $\gamma \in (0,1)$ such that
\begin{align*}
    & e^{\gamma |y_2|}\nabla_\vec{y} \vec \gamma^{j,i,bl} \in L^2(Z^+ \cup Z^-)^4, \ e^{\gamma |y_2|} \vec \gamma^{j,i,bl} \in L^2(Z^-)^2, 
    \ e^{\gamma |y_2|} \pi^{j,i,bl} \in L^2(Z^-) .
\end{align*}
We introduce $
\vec \gamma^{j,i,bl,\varepsilon}(\vec x) = \varepsilon \vec \gamma^{j,i,bl}(\vec y), \ \pi^{j,i,bl,\varepsilon}(\vec x) = \pi^{j,i,\varepsilon}(\vec y)\text{ for } \vec x \in \Omega^\varepsilon
$
and extend $\vec \gamma^{j,i,bl,\varepsilon}$ by zero for $\vec x \in \Omega \setminus \Omega^\varepsilon$. After~\cite{Jaeger_Mikelic_96}, there exist $C_\pi^{j,i,bl}$ and $\vec C^{j,i,bl}$ such that
\begin{align*}
    & |\vec \gamma^{j,i,bl} - \vec C^{j,i,bl} | + |\pi^{j,i,bl} - C_{\pi}^{j,i,bl}| \leq Ce^{-\gamma y_2} \quad \text{for} \; y_2 > 0.
\end{align*}
For further details, e.g. existence and uniqueness results of the solution to the boundary layer problem~\eqref{eq:BL3}, we refer the reader to \cite{Carraro_etal_15,Jaeger_Mikelic_96}. 
Taking into account problems~\eqref{eq:aux-func} and \eqref{eq:BL3}, we obtain the following velocity and pressure error functions
\vspace{-3ex}
\begin{align*}
     \vec U^{4,\varepsilon} &= \vec U^{3,\varepsilon} - \mathcal{H}(-x_2)\varepsilon^2 \! \sum_{i,j=1}^2 \! \vec \gamma^{j,i,\varepsilon} \frac{\partial^2 p^\PM}{\partial x_i x_j} \notag
     \\
     & \quad - \varepsilon^2 \! \sum_{i,j=1}^2  \! \left( \vec \gamma^{j,i,bl,\varepsilon} \! - \! \mathcal{H}(x_2)  \varepsilon \vec  C^{j,i,bl} \right)\frac{\partial^2 p^\PM}{\partial x_i x_j}\bigg|_\Sigma,
    \\[-1ex]
     P^{4,\varepsilon} &=  P^{3,\varepsilon} - \varepsilon^2 \sum_{i,j=1}^2 \left( \pi^{j,i,bl,\varepsilon} - C_\pi^{j,i,bl} \right)\frac{\partial^2 p^\PM}{\partial x_i x_j}\bigg|_\Sigma.
\end{align*}

\paragraph{b) Compressibility effects coming from the term with $t_1^{bl,\varepsilon}$}

To eliminate the next problematic term in~\cref{eq:divergence}, we consider the following boundary layer problem proposed in~\cite[section~1.2.8]{Jaeger_Mikelic_96}:
\begin{subequations}
\begin{align}
    \operatorname{div}_\vec{y} \vec \zeta^{bl} &=  t_1^{bl} (\vec y) - \mathcal{H}(x_2)N_1^{bl} \quad \text{ in } Z^+ \cup Z^-,
    \\
    \llbracket \vec \zeta^{bl} \rrbracket_S &= - \left( \int_{Z^{bl}}  (t_1^{bl} (\vec y) - \mathcal{H}(x_2)N_1^{bl}) \ \text{d} \vec y \right) \vec e_2 \quad \text{ on } S, \label{eq:zeta-int}
    \\
    \vec \zeta^{bl} &= \vec 0 \quad \text{ on } \cup_{k=1}^{\infty}(\partial Y_s - (0,k)), \quad \vec \zeta^{bl} \text{ is 1-periodic in $y_1$}.
    \end{align}
\end{subequations}
For existence and uniqueness results we refer the reader  to~\cite{Jaeger_Mikelic_96}. 

Correction of the velocity $\vec v^{4,\varepsilon}_{\text{approx}}$ using $\vec \zeta^{bl}$ leads to an additional disturbing term at the interface $\Sigma$. To eliminate this term, we construct a counter flow, i.e. consider $\{\vec v^{\cf}, p^{\cf}\}$ satisfying the Stokes equations~\eqref{eq:macro-ff-a} and the periodicity condition in~\cref{eq:macro-ff-b-periodic}. This counter-flow system is completed with the following conditions on the interface and the upper boundary
\begin{align}
    &v^{\cf}_1 = \llbracket \zeta_1^{bl} \rrbracket_S \frac{\partial}{\partial x_1} \frac{\partial v_1^\FF}{\partial x_2}\bigg|_\Sigma, 
    \qquad
    v^{\cf}_2 = \llbracket \zeta_2^{bl} \rrbracket_S \frac{\partial}{\partial x_1} \frac{\partial v_1^\FF}{\partial x_2}\bigg|_\Sigma \quad \text{ on } \Sigma,
    \\
    &v^{\cf}_1 = 0, 
    \qquad
    v^{\cf}_2 = v^{\cf,\text{in}}_2 \quad \text{ on } \{x_2=h\},\label{eq:counterflow-1-interface}
\end{align}
where the compatibility condition 
\begin{align*}
    \int_0^L v^{\cf,\text{in}}_2 \text{d} x_1 =
    \int_0^L \llbracket \zeta_2^{bl} \rrbracket_S \frac{\partial}{\partial x_1} \frac{\partial v_1^\FF}{\partial x_2}\bigg|_\Sigma \text{d} x_1
    \label{eq:counterflow-1-top}
\end{align*}
has to be fulfilled. 
Therefore, the new error functions read
\begin{subequations}\label{eq:error-4b}
\begin{align*}
    \vec U^{5,\varepsilon} =& \vec U^{4,\varepsilon}
    - \varepsilon^2 \vec \zeta^{bl}(\vec y) \frac{\partial}{\partial x_1} \frac{\partial v_1^\FF}{\partial x_2}\bigg|_\Sigma 
     + \varepsilon^2 \mathcal{H}(x_2)  \vec v^{\cf},
    \quad
     P^{5,\varepsilon} =  P^{4,\varepsilon}  + \varepsilon^2 \mathcal{H}(x_2)p^{\cf}.
\end{align*}
\end{subequations}

\paragraph{c) Compressibility effects coming from the term with $\beta_1^{j,bl,\varepsilon}$} 
The 
com-pressibility effects coming from the boundary layer problem~\eqref{eq:BL-beta} are small and do not require any correction. Using~\cref{eq:BL-beta-estimates} we obtain the following estimate 
\begin{align}
    \bigg| \int_{\Omega^\varepsilon} \varepsilon^2 \sum_{j=1}^2 \left( {\beta_1}^{j,bl, \varepsilon} - \mathcal{H}(x_2) M_1^{j,bl} \right) \frac{\partial}{\partial x_1} \frac{\partial p^\PM }{\partial x_j}\bigg|_\Sigma \bigg|
    \leq C\varepsilon^{5/2}.
\end{align}

\paragraph{d) Compressibility effects coming from the term with $q_1^{j,bl}$}
As the last step, we correct the contribution to the divergence~\eqref{eq:divergence} coming from the outer  boundary layer problem~\eqref{eq:BL-q}. For this purpose, we consider the boundary layer corrector proposed in~\cite{Carraro_etal_15,Jaeger_Mikelic_96} which satisfies
\begin{subequations}\label{eq:BL3d}
\begin{align}
 \operatorname{div}_\vec{y} \vec Z^{j,bl} &=  q_1^{j,bl}\quad \text{in }  Z^- ,
 \\
 \llbracket 
 \vec Z^{j,bl}  \rrbracket_S 
 &= - \left( \int_{Z^-} q_1^{j,bl} \ \text{d} \vec y\right) \vec e_2 \quad  \text{on } S ,\\
 \vec Z^{j,bl} &= \vec 0 \quad \text{ on } \cup_{k=1}^{\infty}(\partial Y_s - (0,k)), \quad \vec Z^{j,bl} \text{ is 1-periodic in $y_1$}.
\end{align}
\end{subequations}
After~\cite{Jaeger_Mikelic_96}, there exists at least one $\vec Z^{j,bl} \in H^1(Z^+ \cup Z^-)^2 \cap C_\text{loc}^\infty(Z^+ \cup Z^-)^2.$
Moreover, $\vec Z^{j,bl} \in W^{1,q}((0,1)^2)$, $\vec Z^{j,bl} \in W^{1,q}(Y-(0,1)^2)$ $\forall q \in [1, \infty)$ and there exists a constant $\gamma \in(0,1)$ such that $e^{\gamma |y_2|} \vec Z^{j,bl} \in H^1(Z^+ \cup Z^-)^2$. Using this corrector, we obtain the following velocity and pressure error functions
\begin{subequations}\label{eq:error-4d}
\begin{align*}
     \vec U^{6,\varepsilon} &= \vec U^{5,\varepsilon} \!-\! \varepsilon^2 \sum_{j=1}^2 \frac{\partial}{\partial x_1} \frac{\partial p^\PM}{\partial x_j}(x_1,-H) \left( 
     \vec Z^{j,bl,\varepsilon} + \varepsilon R^\varepsilon(\vec e_2) \int_{Z^-} \!q_1^{j,bl}  \text{d} \vec y \right), \;\;
     \\
     P^{6,\varepsilon} &=  P^{5,\varepsilon}, 
\end{align*}
\end{subequations}
where $R^\varepsilon$ is the same restriction operator as in~\cite[section~4.5]{Carraro_etal_15} and 
\[
\vec Z^{j,bl,\varepsilon}(\vec x) = \varepsilon \vec Z^{j,bl}\left(\frac{x_1}{\varepsilon}, -\frac{x_2+H}{\varepsilon}\right) \quad \text{for } \vec x \in \Omega^\varepsilon.
\]
The term $\varepsilon R^\varepsilon(\vec e_2) \int_{Z^-} q_1^{j,bl} \ \text{d} \vec y $ corrects the value of  $\vec Z^{j,bl,\varepsilon}$ on $\{x_2=-H\}$.

\begin{corollary}\label{cor:3}
 It holds $\vec U^{6,\varepsilon} \in V_\text{per}(\Omega^\varepsilon)$ and $\norm{\operatorname{div} \vec U^{6,\varepsilon} } \leq C\varepsilon^{5/2}$.
\end{corollary}
\begin{proof}
By construction $\vec U^{6,\varepsilon} \in V_\text{per}(\Omega^\varepsilon)$. We have
\vspace{-0.25ex}
\begin{align}
    \diver & \vec U^{6,\varepsilon}
    = -\varepsilon^2 \sum_{j=1}^2 \left( {\beta_1}^{j,bl, \varepsilon} - \mathcal{H}(x_2) M_1^{j,bl} \right) \frac{\partial}{\partial x_1} \frac{\partial p^\PM }{\partial x_j}\bigg|_\Sigma \! \!
    + \frac{\mathcal{H}(-x_2)}{|Y_\text{f}|} \underbrace{\sum_{j=1}^2 \varepsilon^2  k_{ij}  \frac{\partial^2 p^\PM}{\partial x_i x_j} }_{=\diver \vec v^\PM = 0} 
    \notag
    \\
    &- \varepsilon^3 \sum_{j=1}^2 \left(\vec \gamma^{j,i,bl}(\vec y) - \mathcal{H}(x_2) \vec C^{j,i,bl}  \right) \vdot \nabla \frac{\partial^2 p^\PM}{\partial x_i x_j}\bigg|_\Sigma
    - \varepsilon^2 \zeta_1^{bl}(\vec y)  \frac{\partial}{\partial x_1} \frac{\partial v_1^\FF}{\partial x_2}\bigg|_\Sigma
    \notag
    \\
    &- \varepsilon^3 \sum_{j=1}^2 \nabla \frac{\partial}{\partial x_1} \frac{\partial p^\PM}{\partial x_j} (x_1, -H) \vdot  \! \left( \vec Z^{j,bl} \!  \! \left(\frac{x_1}{\varepsilon}, -\frac{x_2+H}{\varepsilon}\right)  \!  \! +  \!
    R^\varepsilon(\vec e_2)  \! \int_{Z^-} \! \! \! q_1^{j,bl}  \text{d} \vec y\right) \!. 
    \label{eq:cor:3-proof}
\end{align}
Applying the Poincar\'{e} inequality and using the result from~\cite[equation (1.78)]{Jaeger_Mikelic_96} to estimate the divergence, we complete the proof.
\end{proof}

\begin{corollary}\label{cor:4}
For all  $\vec \varphi \in V_\text{per}(\Omega^\varepsilon)$ the following inequality holds
\begin{align}\label{eq:cor:4}
    \bigg| \int_{\Omega^\varepsilon} \nabla \vec U^{6,\varepsilon} & \colon \nabla  \vec  \varphi - \int_{\Omega^\varepsilon} P^{6,\varepsilon} \operatorname{div} \vec \varphi 
    \bigg| 
    \leq C\varepsilon^{3/2} \left( \norm{\nabla \vec \varphi}_{L^2(\Omega^\varepsilon)^4}
    + \norm{\vec \varphi}_{H^1(\Omega_\FF)^2} \right). 
\end{align}
\end{corollary}
\begin{proof}
We consider the weak formulation for the error functions $\vec U^{6,\varepsilon}, \ P^{6,\varepsilon}$:
\vspace{-0.25ex}
\begin{align}
    \int_{\Omega^\varepsilon}&   \nabla \vec U^{6,\varepsilon} \colon \nabla  \vec  \varphi
    - \int_{\Omega^\varepsilon} P^{6,\varepsilon} \operatorname{div} \vec \varphi
    =  \int_{\Omega^\varepsilon}  \nabla \vec U^{3,\varepsilon} \colon \nabla  \vec  \varphi
    - \int_{\Omega^\varepsilon} P^{3,\varepsilon} \operatorname{div} \vec \varphi \notag
    \\[-0.5ex]
    & -\int_{\Omega^\varepsilon_\PM} \varepsilon^2  \sum_{i,j=1}^2  \left( \nabla_\vec{y} \vec \gamma^{j,i} (\vec{y}) \frac{\partial^2 p^\PM}{\partial x_i x_j}
    + \vec \gamma^{j,i,\varepsilon} \nabla \frac{\partial^2 p^\PM}{\partial x_i x_j}
    \right) \colon \nabla  \vec  \varphi \notag
    \\[-0.5ex]
    &+\int_{\Sigma} \varepsilon^2  \sum_{i,j=1}^2 \left( \nabla_\vec{y} \vec \gamma^{j,i}(y_1,-0) \frac{\partial^2 p^\PM}{\partial x_i x_j}\bigg|_\Sigma \vec e_2
    -
    \left( 
    \vec \gamma^{j,i,\varepsilon} \otimes \nabla  \frac{\partial^2 p^\PM}{\partial x_i x_j}\bigg|_\Sigma \right) \vec e_2 \right)  \vdot \vec  \varphi  \notag
    \\[-0.5ex]
    & + \int_{\Sigma} \varepsilon^3 \sum_{i,j=1}^2  \left( \left( \vec  C^{j,i,bl}  \otimes \nabla  \frac{\partial^2 p^\PM}{\partial x_i x_j}\bigg|_\Sigma \right) \vec e_2 \right) \vdot \vec  \varphi  
    \notag
    \\[-0.5ex]
    & - \int_{\Omega^\varepsilon} 2 \varepsilon^2 \sum_{i,j=1}^2   \left( \left(\vec \gamma^{j,i,bl,\varepsilon} -  \mathcal H(x_2)\varepsilon \vec  C^{j,i,bl} \right) \otimes  \nabla \frac{\partial^2 p^\PM}{\partial x_i x_j}\bigg|_\Sigma
    \right) \colon \nabla  \vec  \varphi  \notag
    \\[-0.5ex]
    &-\int_{\Omega^\varepsilon} \varepsilon^2 \sum_{i,j=1}^2   \left( \vec \gamma^{j,i,bl,\varepsilon} -  \mathcal H(x_2)\varepsilon \vec  C^{j,i,bl} \right) \frac{\partial^2}{\partial x_1^2} \frac{\partial^2 p^\PM}{\partial x_i x_j}\bigg|_\Sigma  \vdot \vec  \varphi \notag
    \\[-0.5ex]
    & - \int_{\Omega^\varepsilon} \varepsilon^2  \left( \nabla \vec \zeta^{bl}(\vec y) \frac{\partial}{\partial x_1} \frac{\partial v_1^\FF}{\partial x_2}\bigg|_\Sigma
    + \vec \zeta^{bl}(\vec y) \otimes \nabla \frac{\partial}{\partial x_1} \frac{\partial v_1^\FF}{\partial x_2}\bigg|_\Sigma\right) \colon \nabla  \vec \varphi \notag
    \\[-0.5ex]
    &
    - \int_{\Omega^\varepsilon}  \varepsilon^2 \sum_{i,j=1}^2 \left( \pi^{j,i,bl,\varepsilon} - C_\pi^{j,i,bl} \right) \nabla \frac{\partial^2 p^\PM}{\partial x_i x_j}\bigg|_\Sigma  \vdot \vec \varphi \notag
    - \int_{\Sigma} \varepsilon^2 \left( \left( \nabla \vec v^{\cf} - p^{\cf} \ten I\right) \vec e_2 \right) \vdot \vec \varphi
    \\[-0.5ex]
    & - \int_{\Omega^\varepsilon}   \varepsilon^2  \sum_{j=1}^2 \nabla \left( 
    \frac{\partial}{\partial x_1} \frac{\partial p^\PM}{\partial x_j}(x_1,-H) \left( \vec Z^{j,bl,\varepsilon} + \varepsilon R^\varepsilon(\vec e_2) \int_{Z^-} q_1^{j,bl} \ \text{d} \vec y \right) \right) \colon \nabla  \vec  \varphi.
 \label{eq:cor:4-proof}
\end{align}
Using  estimate~\eqref{eq:cor:3-new} for the first two terms on the right hand side of~\cref{eq:cor:4-proof}, we prove inequality~\eqref{eq:cor:4}.
\end{proof}

At this stage, the constructed velocity and pressure approximations are accurate enough. Therefore, we can formulate the interface conditions.

\subsubsection{Leading order approximations and interface conditions}
The leading order velocity and pressure approximations are
\begin{align*}
    \vec{v}^{6,\varepsilon}_\text{approx} =& \mathcal{H}(x_2)\vec{v}^\FF - \mathcal{H}(-x_2) \varepsilon^2 \sum_{j=1}^2 \vec{w}^{j, \varepsilon} \frac{\partial p^\PM }{\partial x_j}  \notag
    - \varepsilon \left(\vec t^{bl,\varepsilon} - \mathcal{H}(x_2) \vec N^{bl} \right) \frac{\partial v_1^\FF}{\partial x_2}\bigg|_\Sigma 
    \\[-0.5ex]
    & 
    + \varepsilon^2 \vec \zeta^{bl}(\vec y) \frac{\partial}{\partial x_1} \frac{\partial v_1^\FF}{\partial x_2}\bigg|_\Sigma \notag
    + \varepsilon^2 \sum_{j=1}^2 \left( \vec \beta^{j,bl,\varepsilon} - \mathcal{H}(x_2) \vec M^{j,bl}\right) \frac{\partial p^\PM }{\partial x_j}\bigg|_\Sigma \notag
    \\[-0.5ex]
    &- \varepsilon^2 \sum_{j=1}^2  \frac{\partial p^\PM }{\partial x_j}(x_1, -H) \vec q^{j,bl}\left(\frac{x_1}{\varepsilon}, -\frac{x_2+H}{\varepsilon} \right) \notag
    - \varepsilon^2 \mathcal{H}(x_2) \vec v^{\cf} + \mathcal{O}(\varepsilon^3),
    \\
    p^{6,\varepsilon}_\text{approx} =& \mathcal{H}(x_2)p^\FF + \mathcal{H}(-x_2) p^\PM
    -  \left(s^{bl,\varepsilon} - \mathcal{H}(x_2) N_s^{bl} \right) \frac{\partial v_1^\FF}{\partial x_2}\bigg|_\Sigma  + \mathcal{O(\varepsilon)}.
    \label{eq:leading-oder}
\end{align*}
Condition~\eqref{eq:pressure-IC-derivation} developed in~\cref{sec:lower-boundary} is the new coupling condition~\eqref{eq:NEW-momentum}. Taking into account the leading order approximations for the velocity $\vec{v}^{6,\varepsilon}_\text{approx}$ 
and the fact that $\vec 0= \llbracket \vec v^\varepsilon \rrbracket_\Sigma \approx \llbracket \vec v^\varepsilon_\text{approx} \rrbracket_\Sigma $, we obtain the remaining coupling conditions~\eqref{eq:NEW-mass} and~\eqref{eq:NEW-tangential}.

\begin{corollary}\label{cor:5}
The following estimate holds true
\begin{align*}
    \norm{\nabla \vec U^{6,\varepsilon}}_{L^2(\Omega^\varepsilon)^4}^2
    \leq& C \varepsilon^{5/2} \norm{P^{6,\varepsilon}}_{L^2(\Omega^\varepsilon)}
    \\
    &+C\varepsilon^{3/2}  \left( \norm{\nabla \vec U^{6,\varepsilon}}_{L^2(\Omega_\PM^\varepsilon)^4}
    + \norm{\vec U^{6,\varepsilon}}_{H^1(\Omega_\FF)^2} \right).
\end{align*}
\begin{proof}
    Substitution of $\vec U^{6,\varepsilon} \in V_\text{per}(\Omega^\varepsilon)$ as a test function in~\cref{eq:cor:4} and use of the results from~\cref{cor:3,cor:4} complete the proof.
\end{proof}
\end{corollary}

\emph{Remark 3.6}.
At this point, we have a similar situation as in~\cite{Carraro_etal_15}. The next step is to estimate $P^{6,\varepsilon}$ using the velocity error function $\vec U^{6,\varepsilon}$. Then, rigorous error estimates for $\vec U^{6,\varepsilon}$ and $P^{6,\varepsilon}$ can be obtained. However, the proof of two-scale convergence for velocity and pressure is beyond the scope of this manuscript.

\section{Model validation}\label{sec:validation}

In this section, we validate the newly derived interface conditions~\eqref{eq:NEW-mass}--\eqref{eq:NEW-tangential} by comparison of macroscale to pore-scale resolved numerical simulations for different pore geometries and flow problems. Besides, the proposed coupling conditions are compared to conditions~\eqref{eq:IC-mass}--\eqref{eq:IC-BJJ} which are typically applied to Stokes--Darcy problems 
even if they are not suitable~\cite{Eggenweiler_Rybak_20}. 
We demonstrate that the Stokes--Darcy model with the newly derived conditions provides accurate results for arbitrary flows to the fluid--porous interface whereas this is not the case for the classical conditions.

\subsection{Discretization and software}

The pore-scale problem~\eqref{eq:pore-scale} is solved using~\texttt{FreeFEM++}\cite{Hecht_12} with the Taylor--Hood (P2/P1) finite elements. The flow domain $\Omega^\varepsilon$ is resolved by approx. 330 000 elements and an adaptive mesh is used. To make the comparison of the pore-scale and macroscale results easier, the pore-scale simulations are averaged (profile: pore-scale). We follow the averaging strategy proposed in~\cite{Lacis_etal_20} and use ensemble averaging to eliminate the microscopic variations. 
Averaging samples are generated by moving the solid inclusions of the porous medium in horizontal direction, hence, pore-scale oscillations in vertical direction still remain in the averaged results. 

The macroscale problem~\eqref{eq:macro-ff} and \eqref{eq:macro-pm} with the classical set~\eqref{eq:IC-mass}--\eqref{eq:IC-BJJ} or the new set~\eqref{eq:NEW-mass}--\eqref{eq:NEW-tangential} of interface conditions is discretized using the second order finite volume method with staggered grids and solved using our in-house \textsc{C++} software. Hereby, the computational domains $\Omega_\FF$ and $\Omega_\PM$ are partitioned into squares with length $h = 10^{-3}$ and the meshes are conforming at the interface~$\Sigma$.

\subsection{Computation of effective properties}
To solve the macroscale problem, effective parameters appearing in the interface conditions and the porous-medium permeability are needed. The permeability tensor $\ten K^\varepsilon$ is obtained by solving the cell problems~\eqref{eq:cell-prob} and  using formula~\eqref{eq:permeability} afterwards. The cell problems~\eqref{eq:cell-prob} are solved by~\texttt{FreeFEM++} using the Taylor--Hood elements and an adaptive mesh with approx.~\mbox{30 000} elements.
To obtain the values of the boundary layer constants the corresponding problems~\eqref{eq:BL-t} and \eqref{eq:BL-beta} are solved. We follow the ideas from~\cite{Carraro_etal_13,Carraro_etal_15} and use the cut-off stripe $Z^{bl}_{l}= Z^{bl} \cap \left((0,1) \times (-l,l)\right)$  for $l=4$ as proposed in~\cite{Carraro_etal_15}.  
These computations are performed also using~\texttt{FreeFEM++} with the Taylor--Hood elements and the fluid part of the cut-off stripe $Z^{bl}_l$ is partitioned into approx.~\mbox{120 000} elements. 

The permeability values and boundary layer constants for two  geometrical configurations considered in this section are presented in~\cref{tab:effective-properties}. 
Note that the values $k_{ij}$ for $i,j=1,2$  have to be scaled by $\varepsilon^2$ to obtain the permeability values as given in~\cref{eq:permeability}. Moreover, the second components of the boundary layer constants vanish,   $M^{j,bl}_2=0$ for $j=1,2$ and $N_2^{bl}=0$. 
The constants in gray presented in~\cref{tab:effective-properties} are zero when solving the cell and boundary layer problems exactly. 
\begin{table}[htbp]
  \caption{Permeability values and boundary layer constants for two different porous-medium geometries.}
  \includegraphics[scale=0.875]{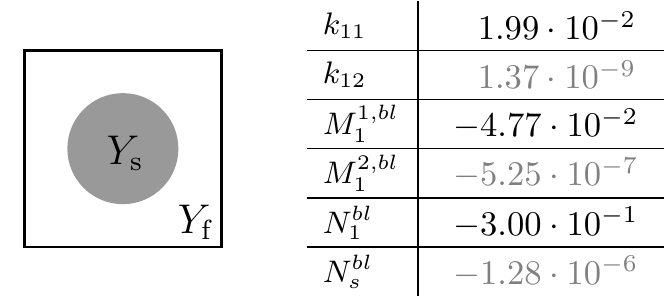} \hspace*{2ex}
  \includegraphics[scale=0.875]{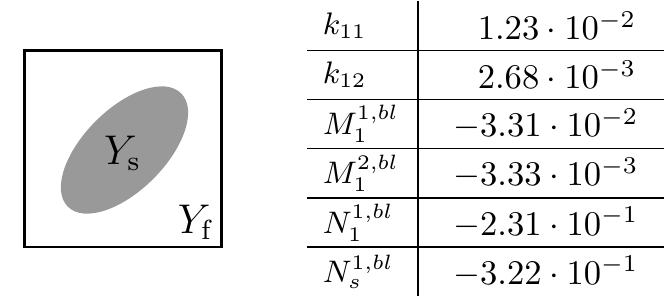} 
  \label{tab:effective-properties}
\end{table}

\subsection{Interface location}

For the macroscale numerical simulations with the classical set of interface conditions the location of the sharp interface $\Sigma$ is uncertain. 
Within this paper, two interface locations (\cref{fig:geometry}b) are investigated: i)~the interface $\Sigma_\textnormal{d}$ is located directly on top of the first row of solid obstacles as proposed in the literature for circular solid inclusions~\cite{Lacis_Bagheri_17,Rybak_etal_19} and ii) the interface $\Sigma_0$ is located on top of the periodicity cells ($x_2=0$) which is the same position as for the new coupling conditions. Numerical simulation results for the Stokes--Darcy problem with the classical interface conditions are labeled according to the interface location (profiles: macroscale,~$\Sigma_{\text{d}}$ or macroscale,~$\Sigma_0$).

The interface position for the newly developed coupling conditions can be set at any distance of order $\mathcal{O}(\varepsilon)$ from the first row of solid inclusions~\cite{Jaeger_etal_01}. This is also the case for the interface $S$ in the boundary layer stripe. We note that a change in the interface position leads to a change in the boundary layer constants only~\cite{Jaeger_etal_01}. For the simulations with the new coupling  conditions~\eqref{eq:NEW-mass}--\eqref{eq:NEW-tangential} the interface $\Sigma_0$ is considered (profile: macroscale, new) as mentioned above.

\subsection{Validation cases}\label{sec:numerical-examples}
In this section, we present three test cases to validate the newly developed interface conditions: two settings with isotropic porous media (validation cases 1 and 2 in \cref{sec:validation-case-1,sec:validation-case-2}) and one with an anisotropic medium (validation case 3 in \cref{sec:validation-case-3}). Validation case~1 
corresponds to the periodic boundary conditions which are needed for the theoretical derivation. However, the
pro-posed 
conditions are not limited to such boundary value problems (validation cases 2 and 3). Further, to demonstrate the advantage of 
the proposed coupling conditions~\eqref{eq:NEW-mass}--\eqref{eq:NEW-tangential}, we compare them to the classical conditions~\eqref{eq:IC-mass}--\eqref{eq:IC-BJJ}.
For all validation cases we consider the free-flow region $\Omega_\FF=(0,1)\times(0,0.5)$ and the porous-medium domain $\Omega_\PM=(0,1)\times(-0.5,0)$ divided by the sharp interface $\Sigma =(0,1)\times \{0\}$. The porous medium consists of $20 \times 10$ solid inclusions which leads to $\varepsilon=1/20$ for all test cases. For the classical interface conditions, we take $\alpha=1$ and $\sqrt{\ten K^\varepsilon} \colon = \sqrt{(\ten K^\varepsilon \vec \tau) \vdot \vec \tau}$.

\subsubsection{Validation case 1}\label{sec:validation-case-1}

In this test case, we consider a flow problem with periodic boundary conditions which corresponds to the theoretical derivation of the new conditions. The porous medium is isotropic, the solid grains are circular with radius $r=0.25\varepsilon$. The effective parameters for this pore geometry are presented in~\cref{tab:effective-properties}.
The boundary conditions for the pore-scale and the macroscale problem are presented in~\cref{fig:setting-validation-case-1}a and the pore-scale velocity field is shown in~\cref{fig:setting-validation-case-1}b. The fluid flow is arbitrary to the fluid--porous interface, especially in the horizontal middle and near the lateral boundaries of the flow domain.

\begin{figure}[!ht]
\centering
    \begin{minipage}{0.05\textwidth} 
        \vspace{-22ex}(a)  
   \end{minipage}
   \hspace{-4.ex}
   \begin{minipage}{0.475\textwidth} 
        \includegraphics[width=0.9\textwidth]{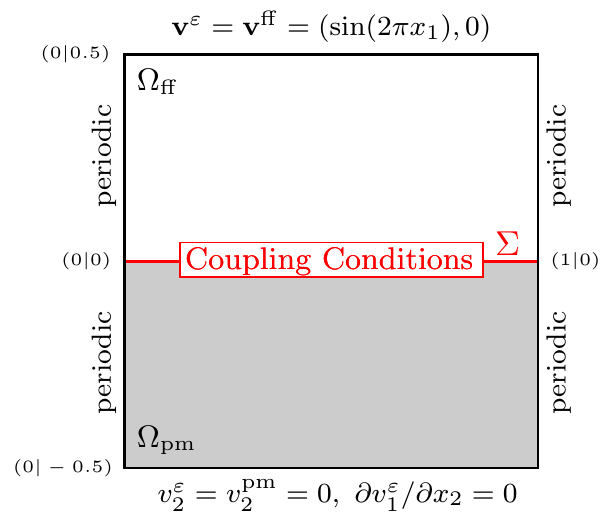} 
   \end{minipage}
   \hspace{-3.5ex}
       \begin{minipage}{0.05\textwidth} 
        \vspace{-22ex}(b)
   \end{minipage}
   \hspace{-1.5ex}
   \begin{minipage}{0.475\textwidth} 
        \includegraphics[width=0.9\textwidth]{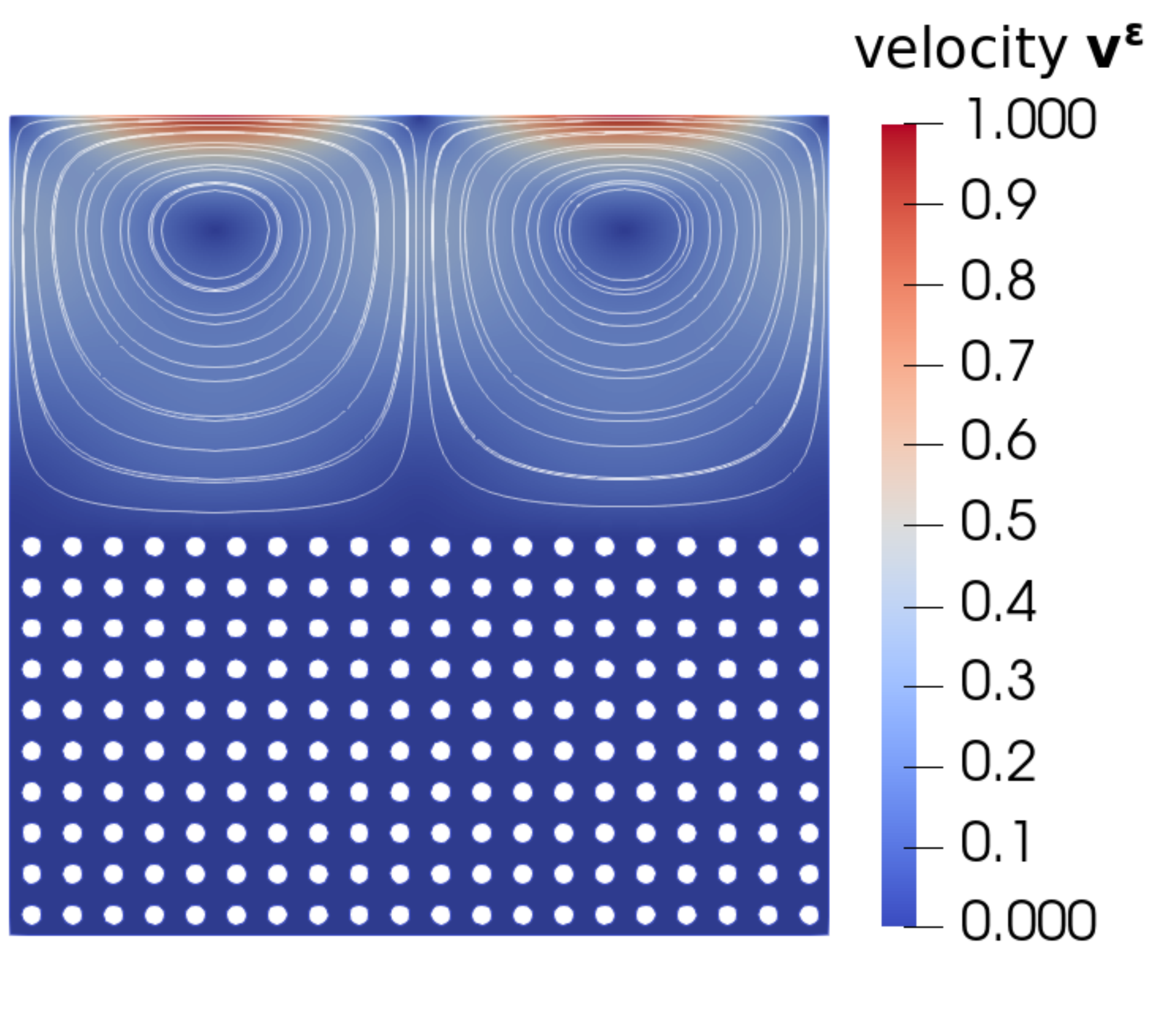} 
   \end{minipage}
   \caption[width=\linewidth]{Flow problem (a) and pore-scale velocity field (b) for validation case~1.}
   \label{fig:setting-validation-case-1}
\end{figure}

\Cref{fig:validation-case-1} provides velocity and pressure profiles for different cross-sections. 
We observe that the profiles of the pore-scale velocity and pressure and the profiles of the macroscale solutions with both classical and new coupling conditions fit well for this case. However, the new conditions provide more accurate results.
\begin{figure}[!ht]
\centering
    \begin{minipage}{0.05\textwidth} 
        \vspace{-30ex}(a)  
   \end{minipage}
   \hspace{-4.5ex}
   \begin{minipage}{0.475\textwidth}
        \includegraphics[width=0.865\textwidth]{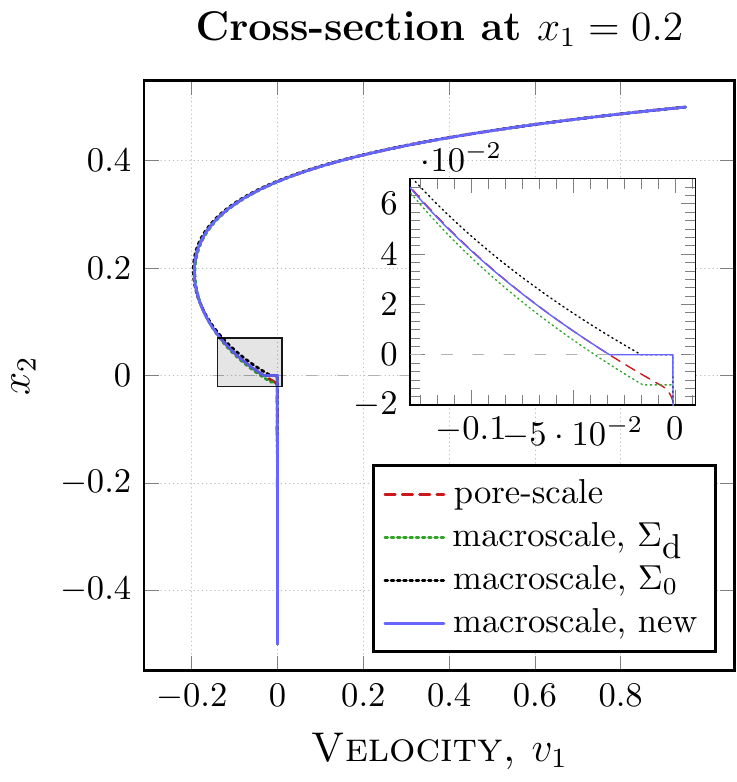}
   \end{minipage}
   \hspace{-1.5ex}
       \begin{minipage}{0.05\textwidth} 
        \vspace{-30ex}(b)
   \end{minipage}
   \hspace{-4.5ex}
   \begin{minipage}{0.475\textwidth} 
        \includegraphics[width=0.865\textwidth]{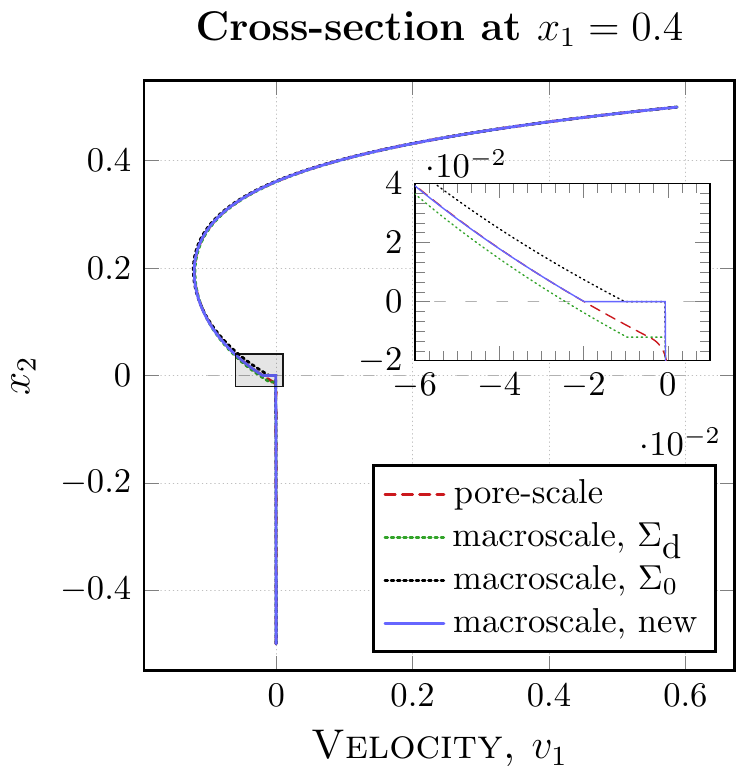}
   \end{minipage}
   \\[1ex]
       \begin{minipage}{0.05\textwidth} 
        \vspace{-30ex}(c)  
   \end{minipage}
   \hspace{-4.5ex}
   \begin{minipage}{0.475\textwidth} 
        \includegraphics[width=0.865\textwidth]{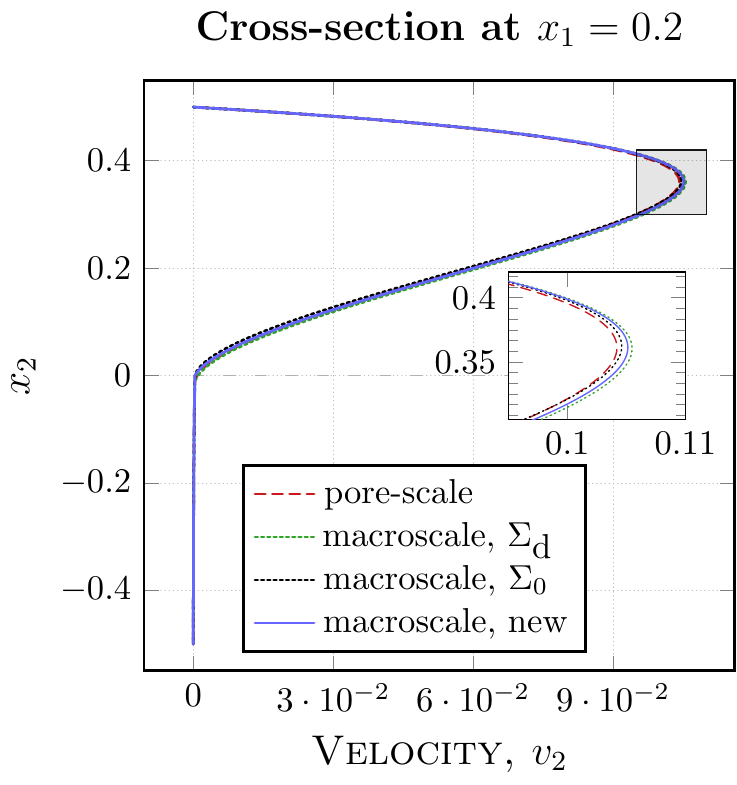}
   \end{minipage}
   \hspace{-1.5ex}
       \begin{minipage}{0.05\textwidth} 
        \vspace{-30ex}(d)
   \end{minipage}
   \hspace{-4.5ex}
   \begin{minipage}{0.475\textwidth} 
        \includegraphics[width=0.865\textwidth]{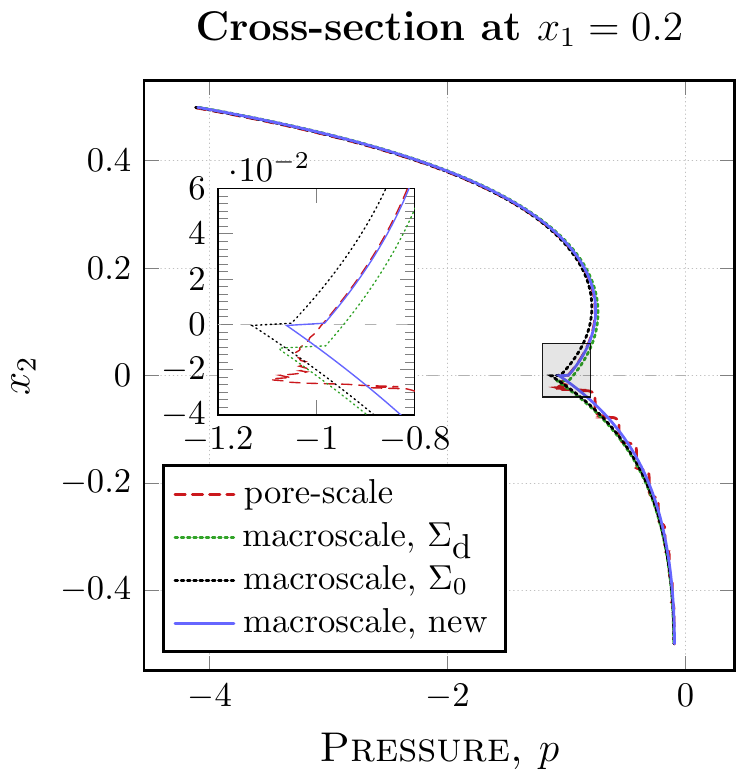}
   \end{minipage}
   \caption[width=\linewidth]{Velocity and pressure profiles for validation case 1.}
   \label{fig:validation-case-1}
\end{figure}

\subsubsection{Validation case 2}\label{sec:validation-case-2}
With this test case, we demonstrate that the newly developed interface
con-ditions 
are valid also for more general coupled problems with fluid flow arbitrary to the porous medium. We consider the same pore geometry as in~\cref{sec:validation-case-1}. Hence, the permeability tensor and the boundary layer constants are identical to those in validation case 1 (\cref{tab:effective-properties}). 

\begin{figure}[htbp]
\centering
    \begin{minipage}{0.05\textwidth} 
        \vspace{-23ex}(a)  
   \end{minipage}
   \hspace{-4.ex}
   \begin{minipage}{0.475\textwidth} 
        \includegraphics[width=0.875\textwidth]{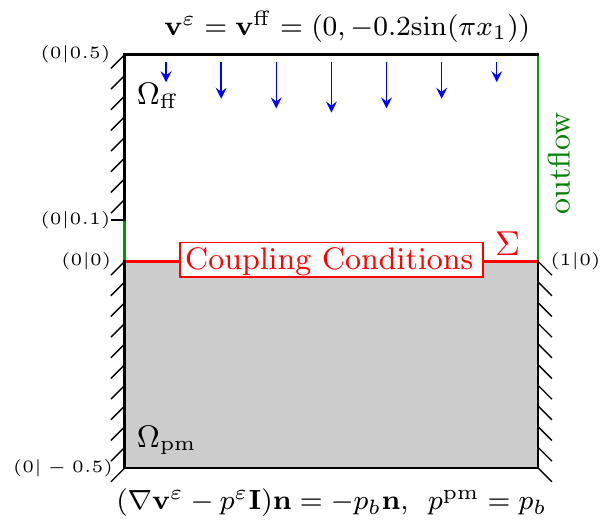}
   \end{minipage}
   \hspace{-3.5ex}
       \begin{minipage}{0.05\textwidth} 
        \vspace{-23ex}(b)
   \end{minipage}
   \hspace{-1.5ex}
   \begin{minipage}{0.475\textwidth} 
        \includegraphics[width=0.85\textwidth]{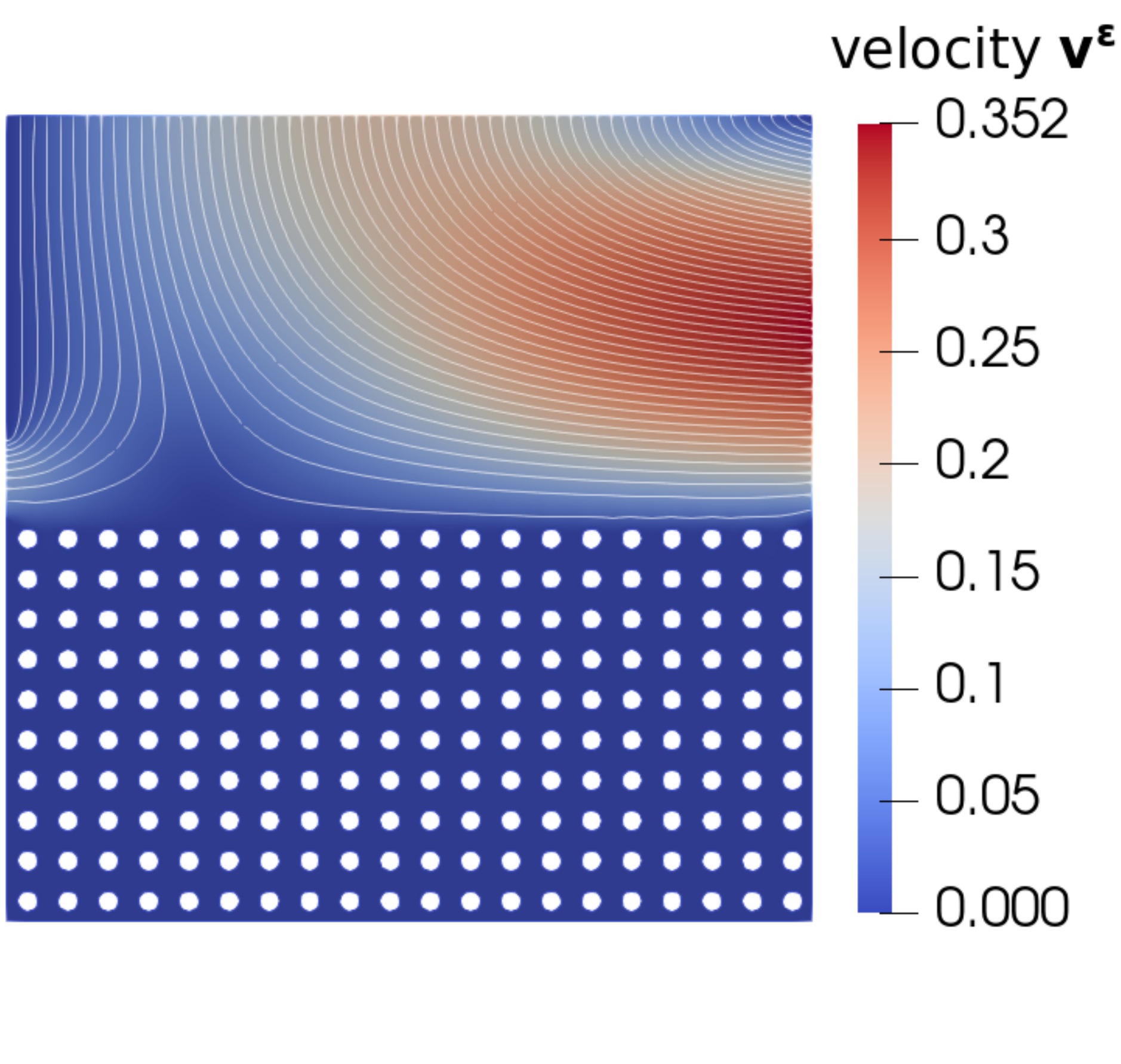}
   \end{minipage}
   \caption[width=\linewidth]{Flow problem for validation cases~2, 3~(a). Pore-scale velocity for validation case~2 (b).}
   \label{fig:setting-validation-case-2}
\end{figure}

The boundary conditions for this validation case are presented in~\cref{fig:setting-validation-case-2}a,  where $p_b=(10^{-6}-x_1)$. The 'outflow' boundary condition is given by $( \nabla \vec v - p\ten I) \vec n=\vec0$, where $\vec v = \vec v^\varepsilon$, $p = p^\varepsilon$ for the pore-scale problem and $\vec v = \vec v^\FF$, $p=p^\FF$ for the macroscale problem. The pore-scale velocity field of this flow system is provided in~\cref{fig:setting-validation-case-2}b. One can observe that within the left part of the flow region ($x_1 < 0.5$) the flow is arbitrary to the fluid--porous interface, whereas in the right part ($x_1 > 0.5$) the flow is almost parallel to the porous medium.

\Cref{fig:validation-case-2} shows velocity and pressure profiles corresponding to
the pore-scale and macroscale problems at different cross-sections. 
We provide profiles for the tangential velocity component at $x_1=0.1$ and $x_1=0.2$ where the flow is non-parallel to the interface. 
Numerical simulations with the classical coupling conditions do not match to the pore-scale results no matter which interface location, $\Sigma_0$ or $\Sigma_{\text d}$, is chosen (\cref{fig:validation-case-2}a and \cref{fig:validation-case-2}b), although the results with $\Sigma_0$ provide a slightly better fitting. In contrast, the velocity profiles computed with the new interface conditions are almost identical to the averaged pore-scale velocity profiles (\cref{fig:validation-case-2}a and \cref{fig:validation-case-2}b).
In~\cref{fig:validation-case-2}c we provide the profiles for the normal component of velocity at $x_1=0.1$. The simulation result with the new interface conditions fits  very well to the averaged pore-scale velocity profile. The velocity computed using the classical set of interface conditions where the interface is $\Sigma_0$, i.e. the same as for the new conditions, are in acceptable agreement to the pore-scale results. However, when the interface is $\Sigma_{\text d}$, the normal velocity has a completely different shape (\cref{fig:validation-case-2}c). In \cref{fig:validation-case-2}d we present the pressure profiles at $x_1=0.2$. The pressure computed with the new set of interface conditions agrees very well to the averaged pore-scale pressure, what is not the case for the classical conditions. We observe, that the choice of $\Sigma_{\text d}$ provides a better agreement for the pressure, whereas for the velocity the interface location $\Sigma_0$ seems to be more suitable. 
Similar observations are obtained for other cross-sections.

\begin{figure}[!ht]
\centering
    \begin{minipage}{0.05\textwidth} 
        \vspace{-30ex}(a)  
   \end{minipage}
   \hspace{-4.5ex}
   \begin{minipage}{0.475\textwidth} 
        \includegraphics[width=0.865\textwidth]{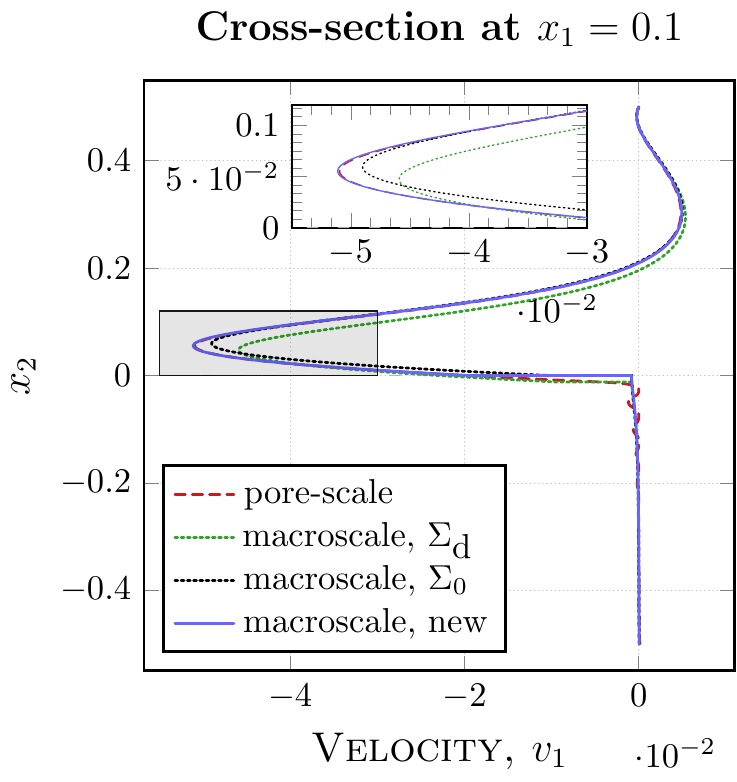}
   \end{minipage}
   \hspace{-1.5ex}
       \begin{minipage}{0.05\textwidth} 
        \vspace{-30ex}(b)
   \end{minipage}
   \hspace{-4.5ex}
   \begin{minipage}{0.475\textwidth} 
        \includegraphics[width=0.865\textwidth]{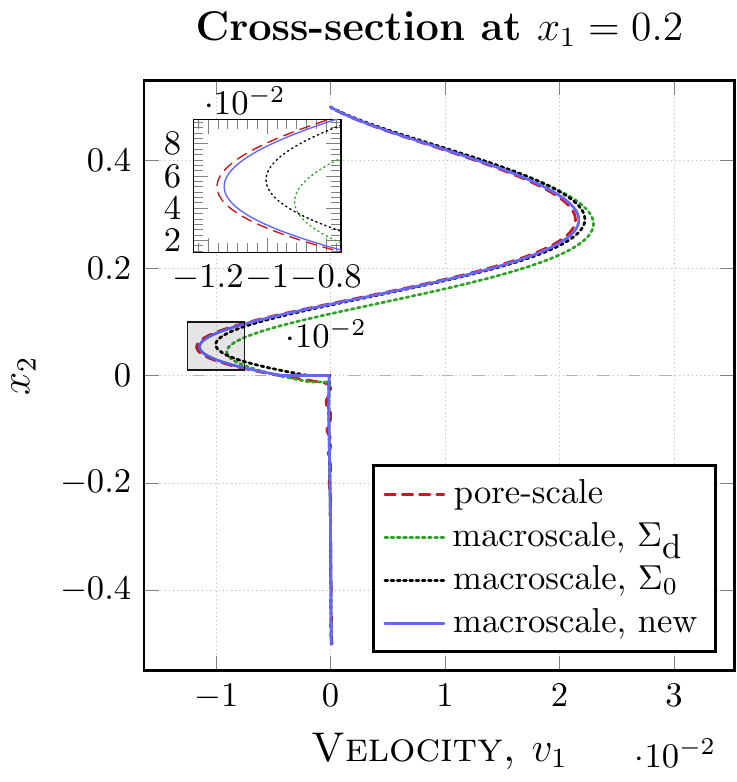}
   \end{minipage}
   \\[1ex]
       \begin{minipage}{0.05\textwidth} 
        \vspace{-30ex}(c)  
   \end{minipage}
   \hspace{-4.5ex}
   \begin{minipage}{0.475\textwidth} 
        \includegraphics[width=0.865\textwidth]{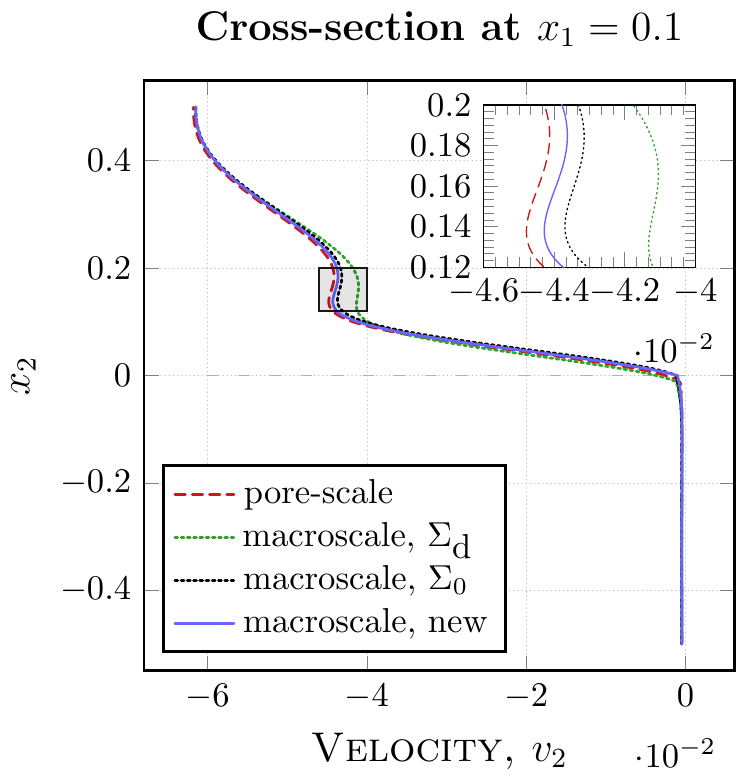}
   \end{minipage}
   \hspace{-1.5ex}
       \begin{minipage}{0.05\textwidth} 
        \vspace{-30ex}(d)
   \end{minipage}
   \hspace{-4.5ex}
   \begin{minipage}{0.475\textwidth} 
        \includegraphics[width=0.865\textwidth]{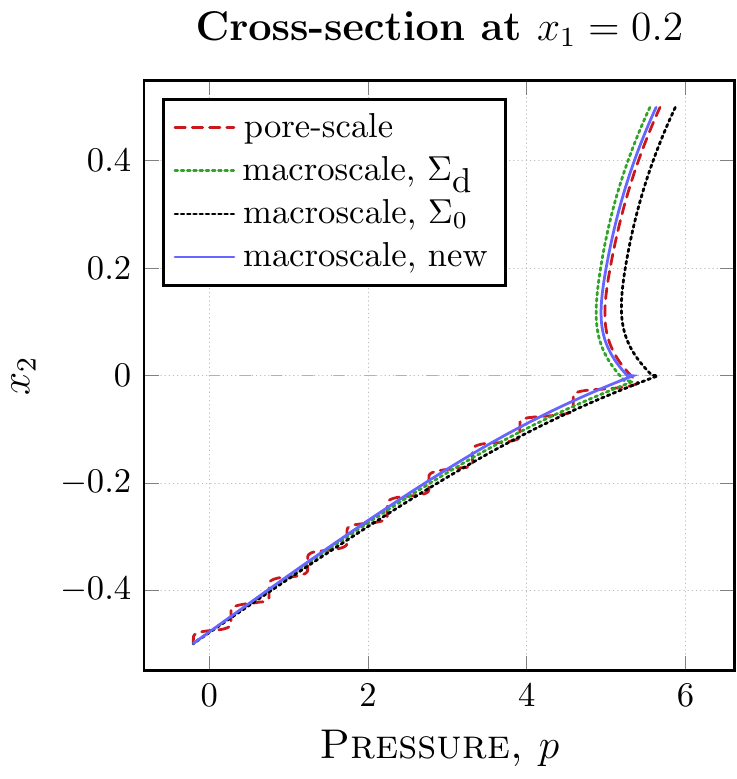}
   \end{minipage}
   \caption[width=\linewidth]{Velocity and pressure profiles for validation case~2.}
   \label{fig:validation-case-2}
\end{figure}

To summarize, the classical coupling conditions fail to represent the flow processes in coupled systems accurately  when the flow is arbitrary to the interface. Different interface locations seem to be correct for velocity and pressure that should not be the case. We note that the value of the Beavers--Joseph parameter $\alpha$ is uncertain as well. In this paper, we consider the most commonly used value in the literature $\alpha=1$.  
More information concerning the choice of $\alpha$  
can be found in~\cite{Eggenweiler_Rybak_20}.
In contrast to the classical coupling conditions, it is evident that the new interface conditions provide results that reflect the pore-scale flow processes accurately.

\subsubsection{Validation case 3}\label{sec:validation-case-3}
The Stokes--Darcy problem is well studied for isotropic porous media (interface location, fitting of parameter $\alpha$, see  e.g.~\cite{Lacis_etal_20,Mierzwiczak_19,Rybak_etal_19}). However, this is not the case for anisotropic media. To demonstrate the advantage of the proposed  interface conditions, we treat anisotropic porous media as well. We consider the geometrical configuration of the porous medium made up of $20 \times 10$ elliptical solid inclusions distributed periodically 
(\cref{fig:validation-case-3}a). In this case, the solid part $Y_\text{s}$ of the unit cell $Y$ consists of the  ellipse 
\[
e(t) = (0.5,0.5) + \operatorname{cos}(0.25\pi)(0.2\operatorname{cos}(t) + 0.4 \operatorname{sin}(t), -0.2\operatorname{cos}(t) + 0.4 \operatorname{sin}(t)), \; 
 \]
$t \in [0,2\pi)$. These inclusions are neither symmetric with respect to the $x_1$- nor
$x_2$-axis. Therefore, we obtain a full permeability tensor $\ten K^\varepsilon$ (\cref{tab:effective-properties}). Such ellipses are also considered in~\cite{Carraro_etal_15}, where the same permeability values are obtained. 
We study the same flow problem as in~\cref{sec:validation-case-2}, i.e. we have the same flow domains and boundary conditions (\cref{fig:setting-validation-case-2}a), only a different porous-medium morphology.

For the classical interface conditions we consider again two interface locations $\Sigma_0$ and $\Sigma_{\text d}$. To the best of our knowledge, there is no recommendation where to locate the interface correctly for non-circular solid inclusions. The interface $\Sigma_0$ is chosen for the newly developed conditions as in \cref{sec:validation-case-1,sec:validation-case-2}.
\begin{figure}[htbp]
\centering
    \begin{minipage}{0.05\textwidth} 
        \vspace{-30ex}(a)  
   \end{minipage}
   \hspace{-2ex}
   \begin{minipage}{0.475\textwidth}
        \includegraphics[width=0.935\textwidth]{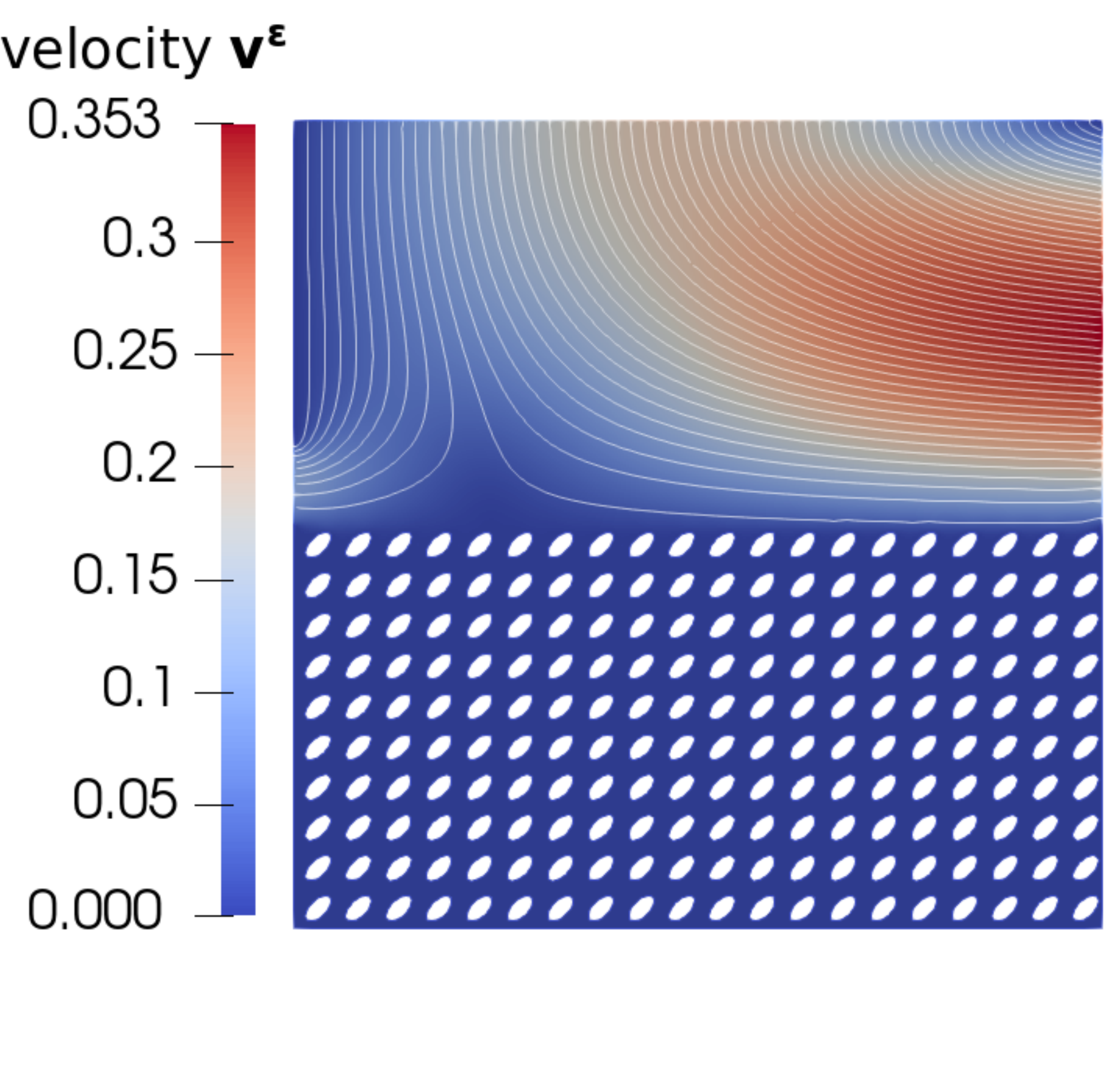}
   \end{minipage}
   \hspace{-0.5ex}
       \begin{minipage}{0.05\textwidth} 
        \vspace{-30ex}(b)
   \end{minipage}
   \hspace{-6ex}
   \begin{minipage}{0.475\textwidth} 
        \includegraphics[width=0.855\textwidth]{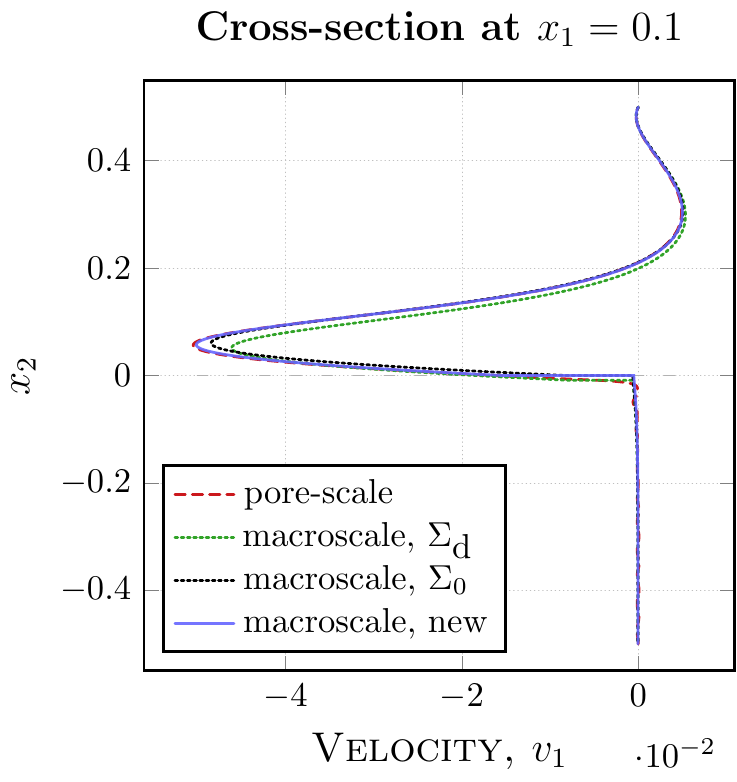}
   \end{minipage}
   \\[1ex]
    \begin{minipage}{0.1\textwidth} 
        \vspace{-30ex}(c)  
   \end{minipage}
   \hspace{-4.5ex}
   \begin{minipage}{0.475\textwidth} 
        \includegraphics[width=0.855\textwidth]{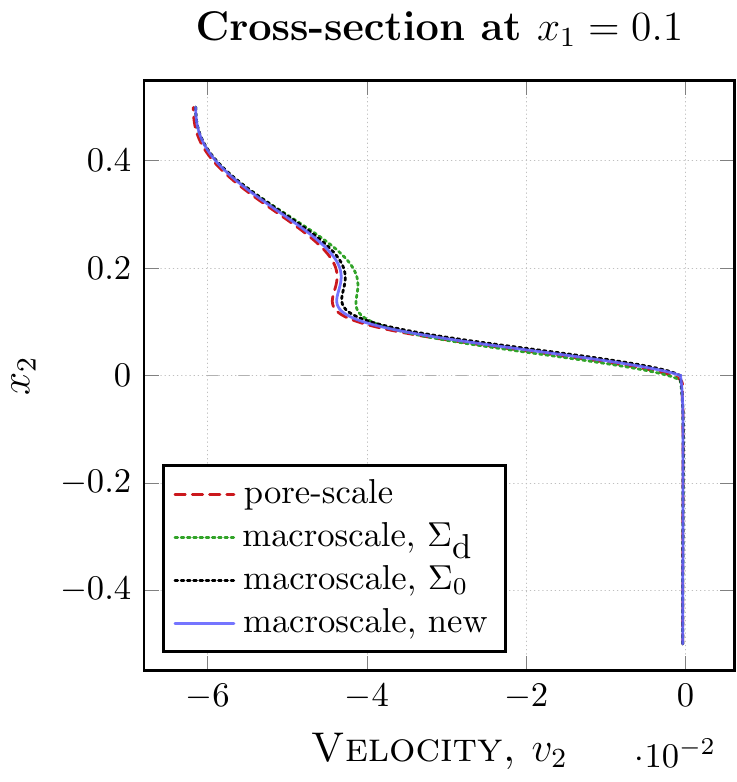}
   \end{minipage}
   \hspace{-2.5ex}
       \begin{minipage}{0.05\textwidth} 
        \vspace{-30ex}(d)
   \end{minipage}
   \hspace{-6ex}
   \begin{minipage}{0.475\textwidth} 
        \includegraphics[width=0.855\textwidth]{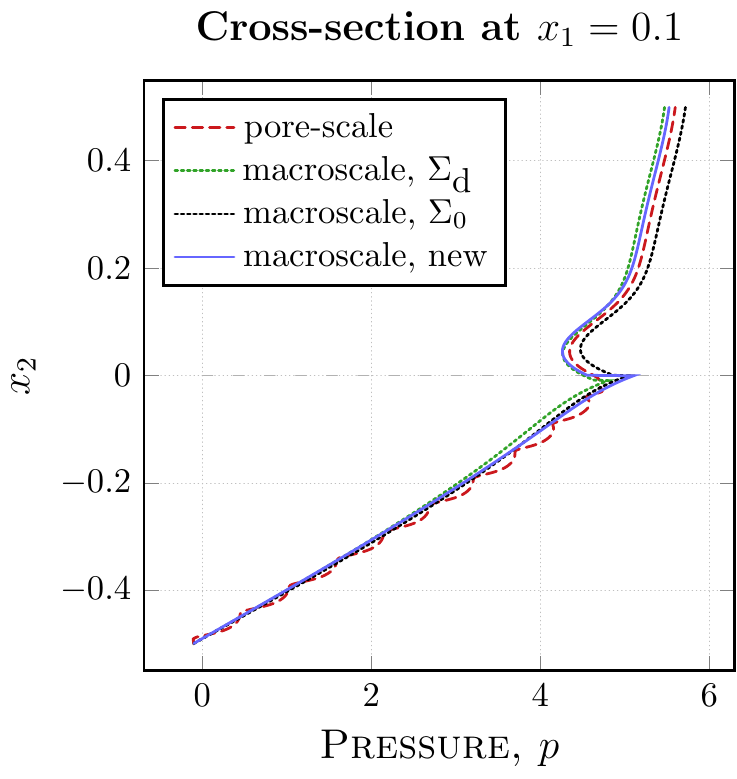}
   \end{minipage}
   \caption[width=\linewidth]{Pore-scale velocity field (a) and velocity and pressure profiles (b)--(d) for validation case 3.}
   \label{fig:validation-case-3}
\end{figure}

In~\cref{fig:validation-case-3} we present profiles of the velocity and pressure at $x_1=0.1$ only, due to similar results as in \cref{sec:validation-case-2}.
We observe that the macroscale simulation results with the new set of interface conditions again agree very well to the pore-scale results, whereas the results with the classical interface conditions do not. The profiles are similar to those presented in~\cref{fig:validation-case-2}, the differences are due to different porous-medium morphology. Note that for this pore geometry, all effective coefficients are nonzero (\cref{tab:effective-properties}).

In conclusion, the three flow problems presented in this section demonstrate the advantage of the newly derived interface conditions~\eqref{eq:NEW-mass}--\eqref{eq:NEW-tangential} over the classical conditions~\eqref{eq:IC-mass}--\eqref{eq:IC-BJJ} to couple free-flow and porous-medium problems with
arbi-trary 
flow direction. 
Moreover, for some flow problems the classical set of coupling conditions is even unsuitable to describe the pore-scale flow processes.

\section{Conclusions}
\label{sec:conclusions}
In this paper, we propose new interface conditions to couple the Stokes equations to the Darcy law for arbitrary flows to the fluid--porous interface. These 
con-ditions are rigorously derived by means of homogenization and boundary layer theory. Starting from the pore-scale perspective we obtained a macroscale  model with new interface conditions. The derived coupling conditions~\eqref{eq:NEW-mass}--\eqref{eq:NEW-tangential} reduce to the ones developed in \cite{Carraro_etal_13,Jaeger_Mikelic_96, Jaeger_Mikelic_00,Jaeger_Mikelic_09,Jaeger_etal_01} using similar  techniques, when the same assumptions on the flow are made. However, our conditions are more general. 

With the new conditions~\eqref{eq:NEW-mass}--\eqref{eq:NEW-tangential} we recovered the mass balance equation~\eqref{eq:IC-mass}, developed an extension of the balance of normal forces across the interface~\eqref{eq:IC-momentum} and an alternative to the Beavers--Joseph coupling condition~\eqref{eq:IC-BJJ}.
The newly developed coupling conditions are validated numerically by comparison of pore-scale to macroscale simulation results. In addition, the proposed interface conditions are compared to the classical ones. It is shown that the classical conditions are less accurate or even unsuitable for Stokes--Darcy problems with arbitrary flows, whereas the new conditions reflect the pore-scale flow processes accurately.

For the new coupling conditions all effective coefficients are computed 
explicit-ly 
based on the pore geometry and the choice of interface position. Thus, no parameter fitting is needed. This is an immense advantage compared to the classical interface conditions where one has to determine the correct interface location and the optimal value of parameter $\alpha$ for each flow problem.
Note that the correct location of the sharp fluid--porous interface is still an open question for the Stokes--Darcy problem with the classical set of interface conditions, especially for anisotropic porous media. 

\section*{Acknowledgments}
The work is funded by the Deutsche Forschungsgemeinschaft (DFG, German Research Foundation) -- Project Number 327154368 -- SFB 1313.

\bibliographystyle{siamplain}
\bibliography{references}
\end{document}